\documentclass[11pt]{smfart}
\usepackage[T1]{fontenc}
\usepackage[francais]{babel}
\usepackage[latin1]{inputenc}
\usepackage{layout}
\usepackage{amstext,amssymb,amscd}
\usepackage{smfthm}

\newtheorem*{Fait}{Fait}

\def\N{{\mathbb N}}
\def\R{{\mathbb R}}
\def\C{{\mathbb C}}

\def\Q{{\mathbb Q}}

\def\A{{\mathbb A}}

\def\tr{{\text{tr}}}

\def\W{{\mathbb W}}
\def\i{\iota}

\def\vhi{{\varphi}}
\def\bk{{\backslash}}
\def\om{{\omega}}

\author{Mathieu Cossutta}
\address{D\'epartement de math\'ematiques et applications\\45, rue d'Ulm \\ F 75230 Paris cedex 05}
\email{Mathieu.cossutta@ens.fr}
\title{Asymptotique des nombres de Betti des vari{\'e}t{\'e}s arithm{\'e}tiques}

\begin{document}
\frontmatter
\begin{abstract}
Nous \'etudions la question de la croissance des nombres de Betti de
certaines vari\'et\'es arithm\'etiques dans des rev\^etements de congruence. Plus
pr\'ecisement nos r\'esultats portent sur les vari\'et\'es de Siegel et les
vari\'et\'es associ\'ees \`a des groupes orthogonaux. Noux expliquons comment un
th\'eor\`eme de Waldspurger permet de majorer et de minorer ces nombres. Les
r\'esultats obtenus vont dans le sens de conjectures de Sarnak et Xue \cite{XS}.
\end{abstract}

\begin{altabstract}
We study the question of the growth of Betti numbers of certain arithmetic
varieties in tower of congruence coverings. In fact, our results are about
Siegel varieties and varieties associated to orthogonal groups. We explain
how a theorem of Waldspurger can be used to obtain lower and upper
bound. Our results are in the direction of conjectures made by Sarnak and
Xue \cite{XS}.
\end{altabstract}


\maketitle

\section{Introduction}
Soit $G^0$ un groupe alg\'ebrique connexe semi-simple sur $\Q$. Soit
\'egalement $K_\infty\subset G^0(\R)$ un compact maximal,
$X=G^0(\R)/K_\infty$ est un espace sym\'etrique. Soit $\Gamma$ un
sous-groupe arithm\'etique de $G^0(\Q)$. On construit ainsi une vari\'et\'e
arithm\'etique $S(\Gamma)=\Gamma\bk X$. Soit $\Gamma(n)$ la suite des sous-groupes de congruence de $\Gamma$, nous sommes int\'eress\'es par le comportement asymptotique des nombres de Betti $L^2$ des vari\'et\'es $S(\Gamma(n))$. Plus pr\'ecisement nous cherchons des nombres $\alpha_{i}$ et $\beta_{i}$ tels que pour tout $\epsilon>0$:
\begin{eqnarray}\label{FormS}
\text{vol}(S(\Gamma(n)))^{\alpha_{i}-\epsilon}\ll_\epsilon\dim H^i(S(\Gamma(n)),\C)\ll_\epsilon\text{vol}(S(\Gamma(n)))^{\beta_i+\epsilon}
\end{eqnarray}
la notation $A_n\ll_\epsilon B_n$ veut dire qu'il existe une constante
$C>0$ ne d\'ependant que de $\epsilon$ telle que $A_n\leq CB_n$ pour tout $n$. Pr\'ecisons les groupes que nous allons \'etudier. Soit $F$ un corps de nombres totalement r{\'e}el. Soit $(E,\overline{\bullet})$ une
extension de corps de $F$ avec une involution de l'un des types suivants:
\[
E=\left\{
\begin{array}{ll}
F& \text{ cas 1}\\
\text{une extension quadratique }\text{ de }F&\text{ cas 2,}
\\
\end{array}\right.
\]
\[
\overline{\bullet}=\left\{
\begin{array}{ll}
\text{id}& \text{ cas 1}\\
\text{l'involution de Galois de }E&\text{ cas 2.}
\end{array}\right.
\]
Soit $\eta\in\{-1,1\}$. Soit $V$ un $E$-espace vectoriel de dimension finie muni d'une forme
sesquilin{\'e}aire $(,)$ $\eta$-hermitienne. Nous faisons
l'hypoth\`ese suivante, dans le cas
$(1)$ nous supposerons que la dimension de $V$ est paire si
$\eta=1$. Soit $G=U(V)$ (resp. $G^0=SU(V)$) le groupe des
isom{\'e}tries de $V$ (resp. de d\'eterminant $1$), on peut voir ces
groupes comme \'etant d\'efinis sur $\Q$ par r\'estriction des
scalaires. On se donne $L$ un r\'eseau entier maximal de $V$. Alors pour
tout id\'eal $\mathfrak{c}$ de $E$ le sous-groupe de congruence
$\Gamma(\mathfrak{c})$ est bien d\'efini. Fixons un tel id\'eal
$\mathfrak{c}$ et $\mathfrak{P}$ un id\'eal premier inerte et non ramifi\'e. Notre r\'esultat
principal est le th\'eor\`eme \ref{ppp}. Par exemple, la cohomologie
holomorphe fortement primitive des vari\'et\'es de Siegel
de rang $n$ n'appara\^it qu'en les degr\'es $p(2n-p)$ avec $p\leq n$, on
obtient alors:

\begin{theo} Si $p<\frac{n-1}{2}$:
\begin{multline}\label{resultatintro}
\text{vol}(S(\Gamma(\mathfrak{c}\mathfrak{P}^k)))^{\frac{p}{n+\frac{1}{2}}(1-\frac{1}{2n})-\epsilon}\ll_\epsilon\\
\dim H^{p(2n-p),0}_2(S(\Gamma(\mathfrak{c}\mathfrak{P}^k),\C)\ll_\epsilon\\
\text{vol}(S(\Gamma(\mathfrak{c}\mathfrak{P}^k)))^{\frac{p}{n+\frac{1}{2}}(1+\frac{p-\frac{1}{2}}{n})+\epsilon}
\end{multline}
\end{theo}
La formule de Matsushima montre que:
\begin{equation}\label{FormM}
 H^i(\Gamma(\mathfrak{c}\mathfrak{P}^k)\bk X)=\oplus_A m\left(A,\Gamma(\mathfrak{c}\mathfrak{P}^k)\right)H^i(\text{Lie}(G^0),K_\infty,A)
\end{equation}
o\`u $A$ d\'ecrit les repr\'esentations unitaires de $G^0(\R)$ et $m(A,\Gamma(\mathfrak{c}\mathfrak{P}^k))$ est
la multiplicit\'e de $A$ dans $L^2(\Gamma(\mathfrak{c}\mathfrak{P}^k\bk G^0)$. Les repr\'esentations v\'erifiant
$$H^*(\text{Lie}(G^0),K_\infty,A)\neq 0$$
sont appel\'ees les repr\'esentations cohomologiques, elles ont \'et\'e d\'etermin\'ees
dans \cite{VZ}. Ces repr\'esentations sont construites \`a partir d'une
alg\`ebre parabolique $\theta$-stable $\mathfrak{q}$ de
$\text{Lie}(G^0)$. On note $A_\mathfrak{q}$ la repr\'esentation associ\'ee
\`a une alg\`ebre $\mathfrak{q}$. On note $R(\mathfrak{q})$ le plus petit
degr\'e tel que $H^i(\text{Lie}(G^0),K_\infty,A_\mathfrak{q})$ est non nul, on a:
$$\dim H^{R(\mathfrak{q})}(\text{Lie}(G^0),K_\infty,A)=1.$$
 Soit $p(\mathfrak{q})$ la limite
inf\'erieure sur les r\'eels $p\geq 2$ tels que les coefficients $K_\infty$-finis de
$A_\mathfrak{q}$ soient dans $L^p(G^0)$. Posons
$d(\mathfrak{q})=\frac{2}{p(\mathfrak{q})}$. Dans \cite{XS} les auteurs
conjecturent que pour tout $\epsilon>0$:
\begin{equation}
m\left(A_\mathfrak{q},\Gamma(\mathfrak{c}\mathfrak{P}^k)\right)\ll_\epsilon [\Gamma:\Gamma(\mathfrak{c}\mathfrak{P}^k)]^{d(\mathfrak{q})+\epsilon}
\end{equation}
g\'en\'eralisant ainsi sous forme conjecturale le th\'eor\`eme de \cite{DGW}. La formule (\ref{FormM}) combin\'ee \`a cette conjecture donne donc une pr\'ediction sur le nombre $\beta_{i}$ souhait\'e dans la formule (\ref{FormS}). Donnons trois exemples supposons que la partie non compacte de $G^0(\R)$ soit:
\begin{itemize}
\item $SU(p,q)$ avec $p>q$, alors pour la cohomologie holomorphe de degr\'e $q$, on trouve $\beta_{q}=\frac{1}{p+q-1}$.
\item $SO(n,1)$, alors pour le $i^\text{i\`eme}$ nombre de Betti, on trouve $\beta_i=\frac{2i}{n-1}$. 
\item $\text{Sp}_{2n}$ alors pour la cohomologie holomorphe $L^2$ de
degr\'e $p(2n-p)$, on trouve $\beta_{p(2n-p)}=\frac{p}{n}.$
\end{itemize}
Signalons que dans \cite{Li4}, l'auteur donne des formules exactes dans le
cas $p=2$. On remarque que
les termes d'erreur par rapport \`a l'exposant $\frac{p}{2n-1}$ dans la
formule (\ref{resultatintro}) tendent vers zero quand $n\gg p$. En
g\'en\'eral, on donne seulement une majoration de
la partie $\theta$ de la cohomologie. Nous allons maintenant d\'efinir
cette partie et expliquer le sh\'ema de notre preuve. Notre m\'ethode est
ad\'elique, dans une premi\`ere partie nous rappelons quelques faits les vari\'et\'es arithm\'etiques dans
ce cadre. Soit $K_f$ un sous-groupe compact ouvert des ad\`eles finies de $G$. On introduit la vari\'et\'e arithm\'etique:
$$S(K_f)=G(\Q)\bk G(\A)/(K\times K_f).$$
La formule de Matsushima ad\'elique (on suppose que
$K_f=\prod_{v\nmid\infty}K_v$) est la suivante:
\begin{eqnarray*}
\dim H^R_\mathfrak{q}(S(K_f),\C)
=\oplus_{A=\otimes A_v}\prod_{v\nmid\infty}\dim A_v^{K_v}.\\
\end{eqnarray*}
La somme portant sur les repr\'esentations automorphes $A\subset
L^2(G(F)\bk G(\A))$ de composante archim\'edienne
${A}_\mathfrak{q}$. Dans \cite{Li1} l'auteur associe \`a toute alg\`ebre
 $\theta$-stable $\mathfrak{q}$ dont le Levi n'a qu'un seul
facteur non compact, un couple $(W_\mathfrak{q},\pi_\mathfrak{q})$ form\'e
d'un espace $-\eta$-hermitien $W_\mathfrak{q}$ sur $E\otimes_F\R$ et d'une
s\'erie discr\`ete de $U(W_\mathfrak{q})$. Le couple v\'erifiant la propri\'et\'e
suivante:
$$\Theta_\chi(\pi_\mathfrak{q})=A_\mathfrak{q}$$
$\Theta_\chi$ d\'esigne la correspondance theta entre la paire de groupes
$U(W)\times U(V)$.
Soit $W$ un espace hermitien d\'efini sur $E$ de signature $W_\mathfrak{q}$
alors on peut d\'efinir:
\begin{eqnarray}\label{FormP1}
H_{\mathfrak{q},W}^R(S(K_f),\C)&=&\oplus_{\pi,\alpha}\dim(\alpha\otimes\Theta(\pi_f))^{K_f}\notag\\
H_{\theta}^R(S(K_f),\C)&=&\oplus_{W,\alpha}\oplus_{\pi}\dim(\alpha\otimes\Theta(\pi_f))^{K_f}
\end{eqnarray}
La somme \'etant ind\'ex\'ee par les repr\'esentations automorphes cuspidales de
$U(W)$ de type $\mathfrak{\pi}_\mathfrak{q}$ \`a l'infini et les caract\`eres
automorphes $\alpha$ de $U(1)$ de composante archim\'edienne nulle. Nous
expliquons ces formules dans la partie \ref{partie2} et \ref{SWF}. Dans la partie \ref{pararam} on d\'emontre que pour
la suite $K(\mathfrak{c}\mathfrak{P}^k)$, les espaces $W$ apparaissant dans
la somme (\ref{FormP1}) sont en nombre finis. Il suffit donc de savoir
borner pour toute place finie $v$ de $F$:
$$\dim\Theta(\pi_v)^{K_v}.$$
Les parties \ref{partie3} et \ref{min} sont consacr\'ees \`a la d\'emonstration des th\'eor\`emes
\ref{maj} et \ref{thmmaj}. Il s'agit de d\'emontrer que si $\mathfrak{P}|p$
est inerte et impaire:
\begin{equation}\label{bornation}
\dim \pi_p^{K(\mathfrak{P}^k)}N\mathfrak{P}^{\frac{k}{2}(n-2m)}\leq
\dim\Theta(\pi_p)^{K(\mathfrak{P}^k)}\leq \dim \pi_p^{K(\mathfrak{P}^{k+1})}N\mathfrak{P}^{\frac{k+1}{2}nm}
\end{equation}
o\`u $n=\dim V$ et $m=\dim W$. La preuve de la majoration est bas\'ee sur des r\'esultats de
Waldspurger (\cite{Wald0} et \cite{Wald}). La preuve de la minoration sur
une version locale de la formule du produit scalaire de Rallis. La partie
\ref{finitude2} d\'emontre une majoration pour les places divisant $2$. Globalement
on conclut en faisant le produit sur les places finies des in\'egalit\'es
(\ref{bornation}) et en utilisant le th\'eor\`eme de \cite{Clo} sur la
multiplicit\'e des s\'eries discr\`etes dans le spectre cuspidal. Ces
r\'esultats peuvent se g\'en\'eraliser \`a de nombreuses autres repr\'esentations
cohomologiques. Nous esp\'erons pouvoir y revenir par la suite.

Mes plus sinc\`eres remerciements vont \`a mon directeur de th\`ese Nicolas
Bergeron. Ces id\'ees ont beaucoup influenc\'e ce
travail. Je remercie \'egalement Jean-Loup Waldspurger pour m'avoir
expliqu\'e avec beaucoup de gentillesse la preuve de la partie
\ref{finitude2}. Enfin je voudrais remercier Benjamin Schraen pour de nombreuses relectures.

\section{Vari\'et\'e arithm\'etique.}\label{partie1}

\subsection{D\'efinition}

Soit $F$ un corps de nombre totalement r{\'e}el. On note
$F_\infty=F\otimes_\Q\R\simeq\R^{[F:\Q]}$. On note $P$ l'ensemble des places de $F$ et pour $v\in P$, $F_v$ la compl{\'e}tion de
$F$ par rapport {\`a} $v$. On note $F_\infty=\prod_{v|\infty}F_v$. On note
$\A$ les ad{\`e}les de $F$ et $\A^f$ les ad{\`e}les finis de $F$. Plus
g{\'e}n{\'e}ralement si $T$ est un ensemble fini de places, on note
$\A_T=\prod_{v\in T}F_v$ et $\A^T$ le produit restreint des $F_v$ pour
$v\not\in T$. On a $\A=\A_T\times \A^T$, en
particulier $\A=F_\infty\times \A^f$. Rappelons que l'on a fix\'e dans
l'introduction un couple $(E,\overline{\bullet})$ form\'e d'une
extension de corps de $F$ et d'une involution de l'un des types suivants:
\[
E=\left\{
\begin{array}{ll}
F& \text{ cas 1}\\
\text{une extension quadratique }\text{ de }F&\text{ cas 2,}
\\
\end{array}\right.
\]
\[
\overline{\bullet}=\left\{
\begin{array}{ll}
\text{id}& \text{ cas 1}\\
\text{l'involution de Galois de }E&\text{ cas 2.}
\end{array}\right.
\]
Dans le cas $(2)$, on note $\epsilon_{E/F}$ le caract{\`e}re de
$\A^\times_F/F^\times$ associ{\'e} {\`a} l'extension quadratique $E/F$ par la th{\'e}orie du
corps de classe (on peut le d{\'e}crire explicitement au moyen du
discriminant et du symbole de Hilbert de $E$). Soit $\eta\in\{-1,1\}$. Soit $V$ un $E$-espace vectoriel de dimension finie muni d'une forme
sesquilin{\'e}aire $(,)$ $\eta$-hermitienne. Nous faisons
l'hypoth\`ese suivante, dans le cas
$(1)$ nous supposerons que la dimension de $V$ est paire si
$\eta=1$, dans le cas $(2)$ on peut supposer sans
perte de g\'en\'eralit\'e que $\eta=1$, nous le faisons. On note $n=\dim_E V$ et $r$ la dimension d'un sous-espace
isotrope maximal d\'efini sur $E$ de $V$. Soit $G=U(V)$ le groupe des
isom{\'e}tries de $V$. Le groupe $G$ est de
rang d\'eploy\'e $r$. Soit $K_\infty$ un sous-groupe compact maximal de
$G(F_\infty)$. Soit $K_f$ un sous-groupe compact ouvert de
$G(\mathbb{A}^f)$. Nous sommes int{\'e}ress{\'e}s par le comportement asymptotique des nombres
de Betti $L^2$ de la vari{\'e}t{\'e} arithm{\'e}tique suivante:
$$\text{S}_{K_f}(V)=G(F)\bk G(\A)/(K_\infty\times K_f)$$
quand le volume de $K_f$ tend vers $0$. Nous supposerons qu'il existe une
unique place archim\'edienne $v_\infty$ telle que $G(F_{v_\infty})$ soit
non compact, on a alors trois familles possibles pour
$G(F_\infty)^\text{nc}=G(F_{v_\infty})$:
\begin{itemize}
\item Dans le cas 1 si $\eta=1$, $O(p,q)$ avec $pq\neq 0$ et $p+q=n$
\item Dans le cas 1 si $\eta=-1$, $\text{Sp}_{2k}$ avec $2k=n$ (on a alors
n\'ecessairement $F=\Q$)
\item Dans le cas 2, $U(p,q)$ avec $pq\neq 0$ et $p+q=n$.
\end{itemize} 
\subsection{Composantes connexes}\label{composante}
En g\'en\'eral, par manque d'approximation forte les vari\'et\'es pr\'ec\'edentes ne
sont pas connexes. Une des composantes connexes est toujours donn\'ee par:
$$G(F)\cap K_f\bk X_G.$$
L'ensemble des composantes connexes est d\'ecrit par les doubles
classes suivantes:
$$\text{Cl}(G,K_f)=G(F)\bk G(\mathbb{A}^f)/K_f.$$
Plus pr\'ecisement si $(\gamma_i)$ est un syst\`eme de repr\'esentant de $\text{Cl}(G,K_f)$, on a:
$$\text{S}_{K_f}(V)=\cup_i \left(G(F)\cap\gamma_iK_f\gamma_i^{-1}\right)\bk X_G.$$
Les ensembles $\text{Cl}(G,K_f)$ v\'erifient les propri\'et\'es suivantes:
\begin{itemize}
\item Dans le cas 1 si $\eta=-1$ le groupe $G$ v\'erifie l'approximation
forte, on a donc pour tout compact $\text{Cl}(G,K_f)=\{1\}$ et si
$K'_f\subset K_f$:
$$[G(F)\cap K_f:G(F)\cap K_f']=[K_f:K_f'],$$
on posera $G^0=G$.
\item Dans le cas 1 si $\eta=1$ le groupe alg\'ebrique $G$ n'est pas
connexe, notons $G^0$ sa composante connexe, on a une suite exacte:
$$1\rightarrow G^0\rightarrow G\stackrel{\det}{\rightarrow}\{\pm
1\}\rightarrow 1.$$
Soit $K(1)$ un compact maximal de $G(\mathbb{A}^f)$. On v\'erifie ais\'ement
que $\det K(1)=\{\pm 1\}(\mathbb{A}^f)$. On suppose que
$K_f=\prod_v K_v$ est distingu\'e dans $K$. Posons $K_f^0=K_f\cap
G^0(\mathbb{A}^f)$. Notons $S$ l'ensemble des places o\`u $\det K_v=1$. En utilisant le d\'eterminant on v\'erifie que:
$$\text{card}\left(\text{Cl}(G,K_f)\right)=\text{card}\left(\text{Cl}(G^0,K_f^0)\right)2^{\text{card} S-1}.$$
\item Dans le cas 2. Soit $G^0$ le groupe d\'eriv\'e de $G$. On a une suite exacte:
$$1\rightarrow G^0\rightarrow G\stackrel{\det}{\rightarrow} U(1)\rightarrow 1.$$
En utilisant que $G^0$ est simplement connexe on v\'erifie que:
$$\text{Cl}(G,K_f)=\text{Cl}(U(1),\det K_f).$$
\end{itemize}

\subsection{Multiplicit\'e des s\'eries discr\`etes.}\label{paramu}
Soit $K_{f,n}$ une suite de sous-groupe compact ouvert de $G(\mathbb{A}^f)$
et posons $K_{f,n}^0=K_{f,n}\cap G^0(\mathbb{A}^f)$. Soit $S$ un ensemble
fini de places de $F$, on suppose que $K_{f,n}\stackrel{S}{\rightarrow}1$
dans le sens de \cite{Clo}. Soit $\pi_\infty^0$ une s\'erie discr\`ete de
$G^0(F_\infty)$. On a alors \cite{Clo}:
$$\text{vol}(K_{f,n}^0)\ll m\left(\pi_\infty^0, L^2_0(G^0(F)\bk G^0(\mathbb{A})/K_{f,n}^0)\right)\ll \text{vol}(K_{f,n}^0)$$
$L^2_0$ d\'esigne le spectre cuspidal. Soit $\pi_\infty$ une extension de
$\pi_\infty^0$ \`a $G(F_\infty)$. On remarque que:
\begin{itemize}
\item Cas 1 $\eta=1$, $G^0=G$.
\item Cas 1 $\eta=-1$, 
\begin{eqnarray}\label{mult1}
m(\pi_\infty,  L^2_0(G(F)\bk
G(\mathbb{A})/K_{f,n}))&=& 2^\alpha m(\pi_\infty^0, L^2_0(G^0(F)\bk
G^0(\mathbb{A})/K_{f,n}^0))
\end{eqnarray}
pour un certain $\alpha$.
\item Cas 2, 
\begin{multline}\label{mult2}
m(\pi_\infty,L^2_0(G(F)\bk
G(\mathbb{A})/K_{f,n}))=\\\text{card}\left(\text{Cl}(U(1),\det K_{f,n}\right)m(\pi_\infty^0,L^2_c(G^0(F)\cap K_{f,n}\bk X_G).
\end{multline}
\end{itemize}

\subsection{Formes harmoniques}
 Nous allons dans cette section d{\'e}crire du point de vue de la
th{\'e}orie des repr{\'e}sentations les formes diff{\'e}rentielles sur
ces vari{\'e}t{\'e}s. Selon une convention usuelle nous noterons les alg{\`e}bres de Lie r{\'e}elles avec un indice $0$ et leur complexifi{\'e} sans indice. On note $\Omega$ l'{\'e}l{\'e}ment de Casimir du centre de l'alg{\`e}bre de Lie de $G(F_\infty)$. Soit $\mathfrak{k}_0\subset \mathfrak{g}_0$, l'alg{\`e}bre de Lie de $K_\infty$. Notons $\mathfrak{p}_0$ le suppl{\'e}mentaire orthogonal pour la forme de Killing de $\mathfrak{k}_0$ dans $\mathfrak{g}_0$. Alors $\mathfrak{p}$ s'identifie au complexifi{\'e} du plan tangent de $\text{S}_{K_f}(V)$ en l'identit{\'e}. La forme de Killing s'identifiant {\`a} la m{\'e}trique sur le plan tangent. Les formes diff{\'e}rentielles $L^2$ de degr{\'e} $i$ sur $\text{S}_{K_f}(V)$ sont donc:
$$\hom_{K_\infty}(\bigwedge^i\mathfrak{p},L^2(G(\Q)\bk G(\A)/K_f))$$ 
On remarque que le centre de l'alg{\`e}bre de Lie de $G(F_\infty)$ agit
({\`a} droite) sur cet espace, donc en particulier l'\'el\'ement de Casimir
$\Omega$. Le lemme de Kuga affirme que l'ensemble des formes
diff{\'e}rentielles harmoniques $L^2$ de degr\'e $i$ est \'egal \`a:
$$H_2^i(\text{S}_{K_f}(V))=\hom_{K_\infty}(\bigwedge^i\mathfrak{p},L^2(G(\Q)\bk G(\A)/K_f)))^{\Omega=0}.$$
On en d{\'e}duit donc que:
$$H_2^i(\text{S}_{K^f}(V),\C)=\oplus_{\pi=\pi_\infty\otimes \pi_f} \hom_{K_\infty}(\bigwedge^i\mathfrak{p},\pi_\infty)\otimes\pi_f^{K_f},$$
la somme {\'e}tant index{\'e}e par les sous-repr{\'e}sentations irr{\'e}ductibles sous l'action de $G(\A)$ de: 
$$L^2(G(F)\bk G(\A))^{\Omega=0}.$$ 
 Les repr{\'e}sentations $\pi_\infty$ v{\'e}rifiant les deux propri{\'e}t{\'e}s:
\begin{itemize}
\item$\hom_{K_\infty}(\bigwedge^*\mathfrak{p},\pi_{\infty})\neq 0$ 
\item l'action du Casimir (une constante not{\'e}e $\pi_\infty(\Omega)$) est triviale, 
\end{itemize}
sont appel{\'e}es les repr{\'e}sentations cohomologiques. Elles ont {\'e}t{\'e} classifi\'ees par Vogan et Zuckermann dans \cite{VZ}. Nous nous fixons maintenant une sous-alg{\`e}bre de Cartan $\mathfrak{t}_0$ de $\mathfrak{k}_0$. Les racines $\Delta(\mathfrak{t}_0,\mathfrak{g})$ sont des {\'e}l{\'e}ments de $i\mathfrak{t}_0^*$. On consid{\`e}re les sous-alg{\`e}bres paraboliques $\mathfrak{q}=\mathfrak{l}\oplus \mathfrak{u}$, o{\`u} $\mathfrak{l}$ est le centralisateur d'un {\'e}l{\'e}ment $X\in i\mathfrak{t}_0$ et $\mathfrak{u}$ est le sous-espace engendr{\'e} par les racines positives de $X$ dans $\mathfrak{g}$. Alors $\mathfrak{u}$ est stable sous l'involution de Cartan, on a donc une d{\'e}composition $\mathfrak{u}=\mathfrak{u}\cap\mathfrak{k}\oplus \mathfrak{u}\cap\mathfrak{p}$. Soit $R(\mathfrak{q})$ la dimension de $\mathfrak{u}\cap\mathfrak{p}$. Le sous-module engendr{\'e} sous l'action de $K_\infty$ par la droite $\bigwedge^{R(\mathfrak{q})}\mathfrak{u}\cap\mathfrak{p}$ dans $\bigwedge^{R(\mathfrak{q})}\mathfrak{p}$ est un sous-espace irr{\'e}ductible not{\'e} $V(\mathfrak{q})$ qui appara{\^\i}t avec multipli\-cit{\'e} $1$ dans $\bigwedge^{R(\mathfrak{q})}\mathfrak{p}$. Vogan et Zuckermann d{\'e}montrent qu'il existe une unique repr{\'e}sentation unitaire irr{\'e}ductible $\pi_\infty$ v{\'e}rifiant:
\begin{itemize}
\item$\hom(V(\mathfrak{q}),\pi_\infty)\neq 0$ 
\item $\pi_\infty(\Omega)=0$.
\end{itemize}
Nous notons $A_\mathfrak{q}$ cette repr{\'e}sentation. Elle v{\'e}rifie de plus:
\begin{itemize}
\item $\hom_{K_\infty}(\bigwedge^j\mathfrak{p},A_\mathfrak{q})=0$ si $j<R(\mathfrak{q})$ 
\item $\hom_{K_\infty}(\bigwedge^j\mathfrak{p},A_\mathfrak{q})=\C$ si $j=R(\mathfrak{q})$.
\end{itemize}
Le sous-espace:
\begin{eqnarray*}
H^{R(\mathfrak{q})}_\mathfrak{q}(\text{S}_{K^f}(V),\C)&=&\oplus_{\substack{\pi=\pi_\infty\otimes\pi_f\\\pi_\infty=A_\mathfrak{q}}}\hom(\bigwedge^{R(\mathfrak{q})}\mathfrak{p},A_\mathfrak{q})\otimes\pi_f^{K_f}\\
&\simeq&\oplus_{\substack{\pi=\pi_\infty\otimes\pi_f\\\pi_\infty=A_\mathfrak{q}}}\pi_f^{K_f}\\
\end{eqnarray*}
est appel{\'e} la partie fortement primitive de type $A_\mathfrak{q}$ de la
cohomologie. La
correspondance theta nous permet de cr{\'e}er des classes de cohomologie
dans la partie fortement primitive associ{\'e}e {\`a} certaines alg{\`e}bres paraboliques. 

\section{Utilisation de la correspondance theta}\label{partie2}
\subsection{Rappel}
On se donne un caract{\`e}re non trivial $\psi$ de $\A_F/F$.
 On se donne {\'e}galement
un caract{\`e}re $\chi$ de $\A_E^\times/E^\times$ v{\'e}rifiant $\chi_{|\A_F^\times}=\epsilon_{E/F}$.  Soit $W$ un $E$-espace vectoriel de dimension finie muni d'une forme
sesquilin{\'e}aire $\langle,\rangle$ de sorte que $W$ soit
$-\eta$-hermitien. Nous faisons
l'hypoth\`ese suivante, dans le cas
$(1)$ nous supposerons que la dimension de $W$ est paire si
$\eta=-1$. On note $m=\dim_E W$. On supposera toujours que $n>m$ et dans le
cas (2) que $n>m+1$. Le $F$-espace vectoriel $\mathbf{W}=W\otimes_E V $ est muni de la forme sympl{\'e}ctique:
$$B(,)=\tr_{E/F}(\langle,\rangle\otimes(,)^\iota).$$
Soit $\text{Sp}(\mathbf{W})$ (resp. $U(W)$) le groupe des
isom{\'e}tries de $\mathbf{W}$ (resp. $W$). La paire de groupes $(U(W),U(V))$ forme une paire duale dans le sens de Howe \cite{Howe}. Notons aussi
$\text{Mp}(\mathbf{W})$ le rev{\^e}tement m{\'e}taplectique de
$\text{Sp}(\mathbf{W})$ qui est une extension:
$$1\rightarrow \text{S}^1\rightarrow \text{Mp}(\mathbf{W})\rightarrow
\text{Sp}(\mathbf{W})\rightarrow 1.$$
D'apr{\`e}s notamment Rao, Perrin et Kudla
\cite{K}, le choix de $\chi$ d{\'e}termine un scindage:
$$i_\chi: U(W)(\A)\times U(V)(\A)\rightarrow Mp(\mathbf{W})(\A).$$
Nous rappelons dans la partie \ref{scindage} quelques propri\'et\'es de cette application
ainsi qu'un r\'esultat de Pan \cite{Pan1}. Ce choix d{\'e}termine une repr{\'e}sentation\par\noindent $(\om_\chi,\ S=\otimes_v S_v)$ de
$U(W)(\A)\times U(V)(\A)$ par restriction de la repr\'esentation
m\'etapl\'ectique de  $Mp(\mathbf{W})(\A)$. On se donne une place $v$ de $F$ et $\pi_v$
une repr{\'e}sentation irr{\'e}ductible de $U(W)(F_v)$. Posons:
$$\Theta_\chi(\pi_v,V)=(\pi_v\otimes \om_\chi)_{U(W)(F_v)}$$
(La notation $\Pi_G$, si $\Pi$ est une repr\'esentation d'un groupe $G$,
d\'esigne les coinvariants de $\Pi$ sous l'action de
$G$). $\Theta_\chi(\pi,V)$ est une repr\'esentation de $U(V)(F_v)$. Dans
\cite{Howe}, il est conjectur\'e que:
\begin{conj}[Howe]\label{conjH} $\Theta_\chi(\pi,V)$ admet un unique
quotient irr\'eductible.
\end{conj}
La proposition a \'et\'e d\'emontr\'ee (\cite{Howe}, \cite{Wald0} et
\cite{Wald}) dans presque tout les cas:
\begin{theo}[Howe et Waldspurger]
La conjecture de Howe est vraie si la place $v$ ne divise pas $2$.
\end{theo}
On notera alors $\theta_\chi(\pi_v,V)$ l'unique quotient irr\'eductible de $\Theta_\chi(\pi_v,V)$.
\subsection{Correspondance theta archim\'edienne}
On choisit un compact maximal $K_\infty^W$ de $U(W)(F_\infty)$ et on notera
$K_\infty$, $K_\infty^V$ pour \'eviter des confusions. Soit $v$ une
place archim\'edienne de $F$. Le th{\'e}or{\`e}me $6.2$ de \cite{Li1} montre que les repr{\'e}sentations $A_\mathfrak{q}$ sont dans l'image de la correspondance theta pour $\mathfrak{q}=\mathfrak{l}\oplus \mathfrak{u}$ de la forme suivante:
\begin{enumerate}
\item $U(W)_v=O(p,q)$ et $U(V)_v=\text{Sp}_{2n}$. On suppose que $p$ et $q$ sont pairs. soit $r=\frac{p}{2}$ et $s=\frac{q}{2}$. On suppose que $r+s\leq n$. Alors $K_\infty^W\simeq U(n)$ et $K_\infty^V\simeq O(p)\times O(q)$. La forme de $\mathfrak{l}_0$ doit {\^e}tre:
 $$\mathfrak{l_0}= \mathfrak{u}(1)^r\times \mathfrak{u}(1)^s\times \mathfrak{sp}_{n-(r+s)}.$$
\item $U(W)_v=U(m,n)$ et $U(V)_v=U(p,q)$ avec $m+n\leq p+q$. On a: $K^W=U(m)\times U(n)$ et $K^V=U(p)\times U(q)$. Soit $r,s,k,l\geq 0$ tels que $r+s=m$, $k+l=n$, $r+l\leq p$ et $s+k\leq q$. La forme de $\mathfrak{l}_0$ doit {\^e}tre:
$$\mathfrak{l_0}= \mathfrak{u}(1)^r\times \mathfrak{u}(1)^l\times\mathfrak{u}(1)^k\times \mathfrak{u}(1)^s\times \mathfrak{u}(p-r-l,q-s-k).$$
\item $U(W)_v=\text{Sp}_{2n}$ et $U(V)_v=O(p,q)$, $p+q>2n$, $K^W=U(n)$ et $K^V=O(p)\times O(q)$, $p+q>2n$. Soit $r\leq\frac{p}{2}$ et $s\leq\frac{q}{2}$ tels que $r+s=n$. La forme de $\mathfrak{l}_0$ doit {\^e}tre:
$$\mathfrak{l_0}= \mathfrak{u}(1)^r\times \mathfrak{u}(1)^s\times
\mathfrak{so}(p-2r,q-2s).$$
\end{enumerate}
Ainsi {\`a} toute repr{\'e}sentation cohomologique $A_\mathfrak{q}$ de $U(V)(F_\infty)$ de ce type, on peut associer un unique couple $(W_\mathfrak{q},\pi_\mathfrak{q})$ tel que $\theta_\chi(V,\pi_\mathfrak{q})=A_\mathfrak{q}$, de plus $\pi_\mathfrak{q}$ {\it est une s{\'e}rie discr{\`e}te}. 
\begin{rema}\label{rmq}
Il existe plusieurs alg{\`e}bres paraboliques ayant le m{\^e}me
$\mathfrak{l}_0$ mais ne donnant pas la m{\^e}me repr{\'e}sentation
$A_\mathfrak{q}$. Cependant dans chacun de ces cas $\mathfrak{l_0}$
d{\'e}termine le degr{\'e} $R(\mathfrak{q})$ introduit pr\'ec\'edemment:
\begin{itemize}
\item Dans le cas (1), $R(\mathfrak{q})=(r+s)(2n-(r+s))$.
\item Dans le cas (2): $R(\mathfrak{q})=pq-(p-(r+l))(q-(r+s))$.
\item Dans le cas (3): $R(\mathfrak{q})=ps+r(q-2s)$.
\end{itemize}
Dans les cas $(1)$ et $(3)$ les vari\'et\'es introduites sont des vari\'et\'es de
Shimura et toute la cohomologie holomorphe primitive est d\'ecrite par la
correspondance theta et provient de groupes compacts.
\end{rema}
\subsubsection{Param\`etre $d(\mathfrak{q})$}
On peut calculer les param\`etres $d(\mathfrak{q})$ \`a l'aide des param\`etres
de Langlands des repr\'esentations $A_\mathfrak{q}$. On peut aussi les obtenir plus
simplement en utilisant que $A_\mathfrak{q}$ appara\^it discr\`etement dans la
repr\'esentation de Weil, il suffit alors d'utiliser le th\'eor\`eme 3.2 de
\cite{Li3}. Posons:
$$d=\left\{\begin{array}{lll}
1\text{ si $V$ symplectique}\\
0\text{ si $V$ orthogonal} \\
\frac{1}{2}\text{ si $V$ unitaire.}  \\
\end{array}\right.$$
On obtient que:
$$d(\mathfrak{q})=\frac{m}{n-2+2d}$$

\subsection{Correspondance globale}

Globalement, on suppose que $W$ a la m{\^e}me signature que $W_\mathfrak{q}$ et on consid{\`e}re $\pi$ une repr{\'e}sentation automorphe cuspidale de
$U(W)(F)\bk U(W)(\A)$ de composante archim{\'e}dienne $\pi_\mathfrak{q}$. Le rel\`evement theta global est l'application:
$$\pi\otimes S_\chi\rightarrow \mathcal{A}(U(V)(F)\bk U(V)(\A))$$
qui {\`a} $f\otimes \vhi$ associe:
\begin{equation}
\theta_{\psi,\chi}(f,\vhi)(h)=\int_{U(W)(F)\bk U(W)(\A)}f(g)\Theta_{\psi,\chi}(g,h,\vhi)dg \label{deftheta}
\end{equation}
o{\`u} $\Theta_{\psi,\chi}(g,h,\vhi)$ est le noyau theta. On note $\Theta(\pi,V)$
l'image de cette application dans les formes automorphes et
$\Theta(\pi,V)_d$ la partie de carr{\'e} int{\'e}grable de $\Theta(\pi,V)$, cette partie est discr\`ete de longueur finie. Soit $\Theta$ la projection:
$$\pi\otimes S_\chi\rightarrow \Theta(\pi,V)_d.$$
Un changement de variable dans la d{\'e}finition (\ref{deftheta}) montre que cette application se factorise en:
$$\Theta:\otimes_v \Theta(\pi_v,V)\rightarrow \Theta(\pi,V)_d.$$
Si cette application est non nulle on obtient des repr\'esentations cohomologiques de
type $A_\mathfrak{q}$. Nous le prouvons en pr{\'e}cisant \cite{KR0}[prop 7.1.2].
\begin{theo}\label{conv}Supposons $\Theta(\pi,V)_d\neq 0$, alors:
\begin{itemize}
\item L'application $\Theta$ se factorise en une application:
$$\otimes_{v|2}\Theta(\pi_v,V)\otimes \otimes_{v\nmid 2}\theta(\pi_v,V)\rightarrow \Theta(\pi,V)_d.$$
\item Si de plus la conjecture de Howe est vraie (conj \ref{conjH}) pour les
repr{\'e}sentations $\pi_v$ pour toute place $v|2$, l'application theta se factorise en un isomorphisme:
$$ \otimes_{v}\theta(\pi_v,V)\rightarrow \Theta(\pi,V)_d.$$
En particulier $\Theta(\pi,V)_d$ est alors irr\'eductible.
\item  Si:
\begin{eqnarray}\label{Weil}
\left\{\begin{array}{cc} 
V\text{ est anisotrope ou}\\
n-r>m+2d-1.
\end{array}\right.
\end{eqnarray}
Les fonctions theta sont de carr\'e int\'egrables.
\end{itemize}
\end{theo}
\begin{proof}
Soit $v_0$ une place telle que $\pi_{v_0}$ v{\'e}rifie la conjecture de Howe. Soit $J_{v_0}$ l'unique sous-espace invariant de 
$\Theta(\pi_{v_0},V)$ tel que le quotient $\Theta(\pi_{v_0},V)/J_{v_0}$ soit irr{\'e}ductible. Soit $x_{v_0}\in J_{v_0}$ et $x^{v_0}\in \otimes_{v\neq v_0}\Theta(\pi_v,V)$. Il s'agit de d{\'e}montrer que $\Theta(x_{v_0}\otimes x^{v_0})=0$
Soit $\Pi\subset \Theta(\pi,V)_d$ une composante irr{\'e}ductible. On note
$p$ la projection orthogonale de $\Theta(\pi,V)_d$ sur $\Pi$. Supposons que $p\circ\Theta(x_{v_0}\otimes x^{v_0})\neq 0$. Il existe alors un tenseur pur $\otimes_v \zeta_v\in\otimes\Pi_v$ tel que le produit scalaire $(p\circ\Theta(x_{v_0}\otimes x^{v_0}),\otimes_v \zeta_v)$ soit non nul. L'application:
$$
\begin{array}{ccc}
\Theta(\pi_v,V)&\rightarrow& \Pi_v\\
 x&\mapsto& (p\circ\Theta(x^{v_0}\otimes x),\zeta^{v_0})
\end{array}
$$
est non nulle (donc surjective) et $U(V)(F_v)$-{\'e}quivariante. Son noyau
est donc $J_v$, on en d{\'e}duit que $p\circ\Theta(x_{v_0}\otimes x^{v_0})$
est nul d'o{\`u} une contradiction. En faisant varier $\Pi$ parmi les
composantes irr{\'e}ductibles de $\Theta(\pi,V)_d$, on en d{\'e}duit les
deux premiers points du th{\'e}or{\`e}me. La preuve du troisi\`eme point est
rappel\'ee dans la partie \ref{SWF}, le crit\`ere (\ref{Weil})
sera appel\'e le crit\`ere de Weil.
\end{proof}
\begin{coro}
$\Theta(\pi,V)_d$ est constitu{\'e}e de repr{\'e}sentations automorphes de type $A_\mathfrak{q}$ {\`a} l'infini.
\end{coro}
Nous d{\'e}montrons \'egalement dans la partie \ref{SWF} un r{\'e}sultat d'injectivit{\'e}:
\begin{lemm}\label{injW} Sous les hypoth\`eses de Weil:\par\noindent
Soit $\pi_1\oplus^\perp\dots\oplus^\perp\pi_r\subset
L^2_0(U(W)(F)\bk U(W)(\A))$, une famille finie de
formes automorphes cuspidales irr{\'e}ductibles,
alors les repr\'esentations $\Theta(\pi_i,V)$ sont deux {\`a} deux orthogonales.
\end{lemm}
La preuve est une application du produit scalaire de Rallis (\'equation (\ref{FPR})). 
Nous noterons $H^R_{\theta,W}(\text{S}_{K^f}(V),\C)\subset
H^R_\mathfrak{q}(\text{S}_{K^f}(V),\C)$ le sous-espace engendr{\'e} par la
correspondance theta des formes automorphes cuspidales de $U(W)$ et
$H^R_\theta(...)$ la somme de ces espaces pour les diff\'erents $W$ possibles et de ces espaces tordus par un caract\`ere de $U(V)$ de composante archim\'edienne triviale. On a d'apr{\`e}s le lemme pr{\'e}c{\'e}dent:
$$H_\theta^R(\text{S}_{K^f}(V))=\oplus_{W,\alpha}\oplus_{\pi=\pi_\mathfrak{q}\otimes\pi_f} \left(\alpha\circ\det\otimes\Theta(\pi,V)\right)_{d,f}^{K^f}$$
la somme {\'e}tant index{\'e}e par les formes automorphes cuspidales de
$U(W)$ avec une composante archim{\'e}dienne {\`a} l'infini isomorphe {\`a}
$\pi_\mathfrak{q}$ et les caract\`eres automorphes $\alpha$ de $1$ (resp. $\{\pm 1\}$, resp. $U(1)$) de
composante archim\'edienne triviale dans le cas (1) $\eta=-1$ (resp. dans le cas (1) $\eta=-1$, resp. dans le cas (2)). 
\begin{rema}$\ $
\begin{itemize}
\item
Les rel\`evements theta locaux et globaux sont parfois nuls
\cite{Wald0}. Nous allons faire, pour obtenir une minoration,
l'hypoth{\`e}se de rang stable qui garantira que les rel\`evements theta
locaux et globaux sont non nuls. De plus, sous cette hypoth{\`e}se la conjecture de Howe est vraie. 
\item D'apr{\`e}s des travaux de Kudla et Millson, pour certains choix de $\mathfrak{q}$, ces classes de cohomologie correspondent aux composantes primitives de combinaisons de cycles
totalement g{\'e}od{\'e}siques naturels. 
\item La premi\`ere somme est bien directe, on peut le prouver en utilisant
le d\'ebut de la preuve de la formule du produit scalaire de Rallis.
\end{itemize}
\end{rema}
Pour $m$ suffisamment petit par
rapport {\`a} $n$, il devrait y avoir {\'e}galit{\'e} entre les deux
espaces:
\begin{eqnarray}\label{conjsur}
\varinjlim_{K^f}H_\theta^R(\text{S}(V)_{K^f},\C)=\varinjlim_{K^f}H_\mathfrak{q}^R(\text{S}(V)_{K^f},\C)
\end{eqnarray}
Ce qui revient donc {\`a} un r{\'e}sultat de surjectivit{\'e} de la
correspondance theta. On parle de rang stable en une place inerte $w$ de
$E$ si $V\otimes E_w$ contient un sous-espace isotrope d\'efinie sur $E_w$ de
dimension $m$. Dans le rang stable local Li prouve dans \cite{Li3} que la
conjecture de Howe pour une repr\'esentation unitaire est vraie et que de
plus le rel\^evement theta est non nul. L'hypoth\`ese suivante:
\begin{eqnarray}\label{rangf}
n>2m+4d-2
\end{eqnarray}
assure que pour toutes places $w$ de $E$, $V\otimes E_w$ est dans le rang
stable (il y a aussi une notion de rang stable pour les places d\'eploy\'ees
cf. \cite{Li2} page 206). Nous parlerons de rang singulier si $r> m$.
\begin{theo}\label{surj}$\ $
\begin{itemize}
\item Supposons (\ref{rangf}) alors $\Theta(\pi,V)$ est non nulle et irr\'eductible.
\item Supposons de plus le rang singulier et $W_\mathfrak{q}$ d\'efini
positif alors (\ref{conjsur}) est vrai.
\end{itemize}
\end{theo}
Les deux points du th\'eor\`eme sont de Li, ils sont d\'emontr\'es dans \cite{Li2}
et \cite{Li5}. Le deuxi\`eme point permet de traiter de mani\`ere compl\`ete la
cohomologie holomorphe de certaines vari\'et\'es de Shimura unitaire et de
Siegel d'apr\`es la derni\`ere remarque du paragraphe \ref{rmq}. Remarquons
que le th\'eor\`eme sous l'hypoth\`ese (\ref{rangf}) d\'ecoulera certainement de la
classification du spectre discret par Arthur au moins dans les cas anisotropes.
\subsection{Majoration de la partie $\theta$}

Nous supposons que $K_f=\prod_{v\nmid\infty}K_v$. On a:
\begin{eqnarray}\label{in1}
\dim H^R_{\theta,\mathfrak{q}}(\text{S}_{K^f}(V),\C)&\leq&\sum_\alpha\sum_{\substack{W \\W\otimes F_\infty\simeq W_\mathfrak{q}}}\sum_{\substack{\pi=\otimes \pi_v\\ \pi_\infty\simeq \pi_\mathfrak{q}}}\dim\left(\alpha\otimes\theta(\pi_f,V)\right)^{K_f}\\
&\leq&\sum_\alpha\sum_{\substack{W \\W\otimes F_\infty\simeq W_\mathfrak{q}}}\sum_{\substack{\pi=\otimes \pi_v\\ \pi_\infty\simeq \pi_\mathfrak{q}}}\prod_{v\nmid\infty}\dim\left(\alpha\otimes\theta(\pi_v,V)\right)^{K_v}\notag
\end{eqnarray}
(la premi\`ere somme \'etant ind\'ex\'ee sur les caract\`eres automorphes de $U(1)$ de
composante archim\'edienne triviale, le troisi\`eme signe somme {\'e}tant index{\'e} sur les repr{\'e}sentations automorphes cuspidales de $U(W)$). Nous introduisons maintenant les compacts dont nous allons mener l'{\'e}tude. Nous rappelons que la correspondance theta d{\'e}pend du choix d'un caract{\`e}re $\psi$ de $\A_F/F$. Soit $\mathfrak{f}$ le conducteur de $\psi\circ\tr_{E/F}$, c'est un id{\'e}al de $E$. On consid{\`e}re deux id{\'e}aux $\mathfrak{f}_V$ et $\mathfrak{f}_W$ de $E$ v{\'e}rifiant $\mathfrak{f}=\mathfrak{f}_V\mathfrak{f}_W$. Pour un r{\'e}seau $L$ de $W$ on d{\'e}finit:
$$L^*=\{x\in W| (x,y)\in \mathfrak{f}_W\}$$
On a $L\subset L^*$ si et seulement si $(L,L)\subset \mathfrak{f}_W$. On
consid{\`e}re un r{\'e}seau $L_W$ maximal pour cette propri{\'e}t{\'e}, on
appelera un tel r{\'e}seau presque autodual (un r\'eseau sera dit
autodual si $L=L^*$). On se donne {\'e}galement un
r{\'e}seau presque autodual de $L_V$ par rapport {\`a} $\mathfrak{f}_V$, on supposera dans les situations globales que $\mathfrak{f}_V=1$ et $\mathfrak{f}_W=\mathfrak{f}$. On
note si $v$ (resp. $w$) est une place finie de $F$ (resp. $E$) $\mathcal{O}_v$
(resp. $\mathcal{O}_w$) l'anneau des entiers de $F_v$ (resp. $E_w$). On
note enfin: 
$$\widehat{\mathcal{O}_E}=\prod_{w\nmid\infty}\mathcal{O}_w$$
Soit $\mathfrak{c}$ un id{\'e}al de $E$, d\'efinissons le compact ouvert
$K(\mathfrak{c})$ de $U(V)(\A^f)$:
\begin{eqnarray*}
K(\mathfrak{c})&=&
\{g\in U(W)(\A^f)|(g-1)L\otimes
\widehat{\mathcal{O}_E}\subset
\mathfrak{c}L\otimes\widehat{\mathcal{O}_E}\text{ et }gL\otimes
\widehat{\mathcal{O}_E}=L\otimes \widehat{\mathcal{O}_E}\}\\
&=&\prod_v K(\prod_{w|v}\mathfrak{c}_w)
\end{eqnarray*} 
ce sont les sous-groupes de congruence standard. On notera plus simplement
$X(\mathfrak{c})$ la vari{\'e}t{\'e} $\text{S}_{K(\mathfrak{c})}(V)$. Nous
noterons \'egalement $X^0(\mathfrak{c})$ l'union des composantes connexes de
$X(\mathfrak{c})$ dont l'image par le d\'eterminant est trivial (cf. paragraphe
\ref{composante}). Les formules (\ref{mult1}) et (\ref{mult2}) permettent
d'\'etudier la dimension de la cohomologie de $X^0(\mathfrak{c})$. Nous
notons $T_{V,\psi,\chi}$ l'ensemble des places de $F$ telles que l'un des
objets $V,\psi,\chi$ soit ramifi\'e au dessus de $v$ ainsi que les places
divisant deux. Soit $v$ une place de $F$ et $w$ une place de $E$ au dessus
de $v$. Nous notons $\varpi_v$ (resp. $\varpi_w$) une uniformisante de
l'anneau des entiers de $F_v$ (resp. $E_w$). Supposons $v\not\in
T_{V,\psi,\chi}$, alors $K_v(1)$ est un compact maximal de $U(V_v)$.
\begin{lemm}\label{Nram} Pour tout place finie $v\not\in T_{V,\psi,\chi}$:
\begin{itemize}
\item Si $\left(\alpha\otimes\Theta(\pi,V)\right)^{K(\varpi_w^k)}\neq 0$
alors $\alpha_{|\det K(\varpi_w^k)}=1$.
\item
Si $\Theta_\chi(\pi,V)^{K_v(1)}\neq 0$ alors $W$ et $\pi$ sont non ramifi{\'es}.
\end{itemize}
\end{lemm}
Le premier s'obtient par une relecture minitieuse des travaux de
Waldspurger. Le fait que $W$ et $\Theta_\chi(\pi,V)$ non ramifi\'e implique $\pi$ non
ramifi\'e est du \`a Howe (\cite{Howe}, th\'eor\`eme 7.1.b). Le fait que
$W$ soit non ramifi\'e est prouv\'e dans la partie \ref{pararam}. Ce lemme
montre que:
\begin{itemize}
\item le nombre de caract\`ere $\alpha$ apparaissant dans la somme (\ref{in1})
est de l'ordre de $\text{Cl}(U(1),\det K_f)$,
\item le nombre de $W$ apparaissant dans
$H^R_{\theta,\mathfrak{q}}(X(\mathfrak{c}))$ est fini born{\'e} par le
nombre d'espace $\epsilon$-h{\'e}rmitien v{\'e}rifiant: $W$ est non
ramifi{\'e} si $\mathfrak{c}_v=1$ et $v\not\in T_{V,\psi,\chi}$ et la signature de $W$ {\`a} l'infini est det\'ermin{\'e}e par $\mathfrak{q}$.
\end{itemize}
On fixe une place inerte non archim\'edienne $v_0$ (on note $w_0$ la place de $E$ au dessus de $v_0$) et $\mathfrak{c}$ un id\'eal de $E$ premier \`a $v_0$. On a:
\begin{multline}
\sum_W\sum_{\pi=\pi_\infty\otimes\pi_f}\dim \theta(\pi_f,V)^{K(\mathfrak{c}\varpi_{w_0}^k)}\ll\\
\dim H_{\theta,\mathfrak{q}}^R(X^0(\mathfrak{c}\varpi_{w_0}^k))\ll\\
\sum_W\sum_{\pi=\pi_\infty\otimes\pi_f}\dim \theta(\pi_f,V)^{K(\mathfrak{c}\varpi_{w_0}^k)}
\end{multline}
(les constantes d\'ependent du choix de $\mathfrak{c}$ et de $v$).
\begin{theo}\label{maj} Soit $\pi$ une repr{\'e}sentation unitaire de
  $U(W)_{v_0}$ alors pour $k$ suffisamment grand: 
$$ \dim \Theta(\pi,V)^{K(\varpi_{w_0}^k)}\leq (\mathbb{N}w_0)^{\frac{k+1}{2}mn} \dim \pi^{K(\varpi_{w_0}^{k+1})}$$
\end{theo}
($\mathbb{N}w_0$ est le cardinal du corps r\'esiduel de $E_{w_0}$).
Pour les places divisant $2$ nous prouvons \'egalement un r{\'e}sultat de ce type mais non explicite:
\begin{prop}\label{Pfin}
Pour toute place finie $u$ de $F$ et tout sous-groupe compact ouvert $K_u$
de $U(V_u)$. Il existe un sous-groupe compact ouvert $K'_u$ de $U(W_u)$ et une constante $C>0$ tels que pour toutes repr{\'e}sentations irr{\'e}ductibles $\pi_u$ de $U(W_u)$: $\dim \Theta(\pi_u,V)^{K_u}\leq C\dim\pi_u^{K'_u}$.
\end{prop}
Le th\'eor\`eme est prouv\'e dans la partie \ref{partie3}. La proposition est un corollaire du th\'eor\`eme pour $v\nmid 2$, pour une place au dessus de
$2$ la d{\'e}monstration du th\'eor\`eme 1.4 du chapitre 5 de \cite{MVW} s'adapte pour obtenir ce
r{\'e}sultat, nous en donnons une preuve dans la partie
\ref{finitude2}. Soit $T$ l'union de $T_{V,\psi,\chi}$ et des places
divisant $\mathfrak{c}$ le tout priv\'e de $v_0$. Soit $\mathfrak{c}'$ un id{\'e}al inversible en dehors de $T$ tel
que $K(\mathfrak{c}')$ v{\'e}rifie l'{\'e}nonc{\'e} de la
proposition \ref{Pfin} par rapport {\`a} $K(\mathfrak{c})$. Alors \`a partir
de (\ref{in1}) et en combinant avec les r\'esultats \ref{Nram} et \ref{Pfin} on obtient:
$$\dim H^R_{\theta,\mathfrak{q}}(X^0(\mathfrak{c}\varpi_{w_0}^k))\leq
C'\sum_{\pi}\dim\theta(\pi_{v_0})^{K(\varpi_{w_0}^k)}\times\dim\pi_T^{K(\mathfrak{c}')}\times\prod_{u\not\in
T\cup\{v\}}\dim \pi_u^{K_u(1)}$$
le dernier terme vaut $0$ si $\pi$ est ramifi{\'e} en une
place $u\not\in T\cup\{v_0\}$, et $1$ dans le cas contraire. En utilisant le th{\'e}or{\`e}me \ref{maj},
on a finalement que:
\begin{theo}\label{thmmin}
$$ \dim H_{W,\theta}^R(X^0(\mathfrak{c}\varpi_{w_0}^k))\ll \mathbb{N}\varpi_{w_0}^{\frac{k+1}{2}mn}m(\pi_{\mathfrak{q}},K(\mathfrak{c}'\varpi_{w_0}^{k+1}))
$$
\end{theo}
La multplicit\'e des s\'eries discr\`etes a \'et\'e \'etudi\'ee dans le paragraphe \ref{paramu}. On en d{\'e}duit une
majoration en termes de volume dans le paragraphe \ref{resultp}. Il est aussi int\'eressant de consid\'erer le cas d'une place d\'eploy\'ee
de $E$ (cela n'a de sens que dans le cas 2) alors il existe deux places conjugu\'ees $w$ et $\overline{w}$ au dessus de
$v$. En utilisant \cite{Min}, on obtient de la m\^eme fa\c{c}on le r\'esultat suivant:
$$\dim
H_{\theta,\mathfrak{q}}^R(X^0(\mathfrak{c}\varpi_{w_0}^k\varpi_{\overline{w_0}}^k))\ll
\mathbb{N}{w_0}^{kmn}\text{vol}\left(K(\varpi_{w_0}^{k}\varpi_{\overline{w_0}}^k)\right)$$

\subsection{Minoration de la partie $\theta$}
On suppose que l'hypoth\`ese (\ref{Weil}) est v\'erifi\'ee et que
$n>\frac{5m}{2}+1$. On se donne $v_0$ une place inerte n'appartenant pas \`a
$T_{V,\psi,\chi}$ et $W$ un espace h\'ermitien de type $W_\mathfrak{q}$ \`a
l'infini, non ramifi\'e en $v$, on donne une minoration de
$H_{W,\theta}^R(X^0(\mathfrak{c}\varpi_v^k))$ pour un $\mathfrak{c}$ assez
grand explicite. Nous donnons deux r\'esultats de minoration qui vont jouer le r\^ole de
la proposition \ref{Pfin} et du th\'eor\`eme \ref{maj}.
\begin{prop}\label{propmaj} Soit $u$ une place finie de $F$. Il exite une constante $c>0$ telle que pour tout $k$ suffisamment grand:
 $$\dim\theta(\pi_u,V)^{K(\varpi^{2k+c})}\geq \dim\pi_u^{K(\varpi^k)}$$
\end{prop}

\begin{theo}\label{thmmaj} Soit $\pi$ une repr\'esentation unitaire de $U(W)_{v_0}$. On a la minoration suivante pour $k$ pair:
$$
\dim\theta(\pi,V)^{K(\varpi_{w_0}^k)}\geq \dim\pi^{K(\varpi_{w_0}^{k})}\mathbb{N}{w_0}^{\frac{k}{2}(n-2m)m}
$$
\end{theo}
La preuve de ces deux r\'esultats locaux est donn\'ee dans la partie
\ref{min} (corollaire \ref{minGro} et \ref{minfin}).

Soit $T$, l'union de $T_{V,\psi,\chi}$ de $v$ et des places de ramification
de $W$. D'apr\`es la proposition pr\'ec\'edente, il existe deux id\'eaux $\mathfrak{c}$ et
$\mathfrak{c'}$ explicites de $E$
premier \`a $v$ et aux places n'appartenant pas \`a $T_{V,\psi,\chi}$ tels que
pour toutes places $t\in T_{V,\psi,\chi}$:
$$\dim\theta(\pi_t,V)^{K(\mathfrak{c})}\geq \dim\pi_t^{K(\mathfrak{c}')}$$

De plus on se souvient que si $u\not\in T$:
$$
\dim\theta(\pi_u,V)^{K(1)}=\left\{
\begin{array}{lll}
1&&\text{ si }\pi_u\text{ non ramifi\'ee.}\\
0&&\text{ sinon.}
\end{array}\right.
$$
On en d\'eduit donc que:
\begin{eqnarray*}
&&\dim H^R_{W,\theta}(X^0(\mathfrak{c}\varpi_{w_0}^k))\\
&\geq&\sum_\pi\dim\theta(\pi_{v_0},V)^{K(\varpi_{w_0}^k)}\dim\left((\otimes_{u\neq v_0}\pi_u)^{K(\mathfrak{c'})\times
K^T(1)}\right)
\end{eqnarray*}
En utilisant le th\'eor\`eme \ref{thmmaj}. On
en d\'eduit le th\'eor\`eme suivant:
\begin{theo}\label{thmm} On a pour tout $k$:
$$
\dim H^R_{W,\theta}(X^0(\mathfrak{c}\varpi_{w_0}^k))\gg m(\pi_\mathfrak{q},K(\mathfrak{c}'\varpi_{w_0}^k))\mathbb{N}{w_0}^{\frac{k}{2}(n-2m)m}$$
\end{theo}
Remarquons que pour obtenir le  r\'esultat pour tout $k$, il suffit de le faire pour le cas pair et d'utiliser les inclusions:
$$K(\varpi_0^{2k+2})\subset K(\varpi_0^{2k+1})\subset K(\varpi_0^{2k}).$$
Nous reformulons maintenant les th\'eor\`emes \ref{thmmin} et \ref{thmm} en termes de volume.
\subsection{Enonc\'e des th\'eor\`emes de minoration}\label{resultp}
Nous donnons le calcul asymptotique des volumes des compacts consid\'er\'es
pr\'ec\'edemment dans les diff\'erents cas. Posons:
\begin{eqnarray}\label{defe}
e_V&=&\left\{
\begin{array}{lll}
1&\text{ si }& V\text{  est antisym{\'e}trique}\\
-1&\text{ si }& V \text{ est sym{\'e}trique}\\
0&\text{ si }& V \text{ est hermitien ou anti-h\'ermitien}
\end{array}\right.
\end{eqnarray}
Soit $v$ une place finie de $F$ alors pour tout $\epsilon>0$:
$$(\mathbb{N}v)^{k\alpha(V)-\epsilon}\ll_\epsilon\text{vol}\left(K_V^0(\varpi_w^k)\right)\ll_\epsilon
(\mathbb{N}v)^{k\alpha(V)+\epsilon}$$
avec $\alpha(V)$ la dimension du groupe $SU(V)$, c'est \`a dire: 
\begin{eqnarray}
\alpha(V)=\left\{
\begin{array}{lll}
\frac{n(n+e)}{2}&\text{ dans le cas 1}\\
\frac{(n-1)(n+2)}{2}&\text{ dans le cas 2 et $v$ ramifi\'ee}\\
n^2-1&\text{ dans le cas 2.}
\end{array}\right.
\end{eqnarray}

On a donc la reformulation en termes de volume des th\'eor\`emes pr\'ec\'edents:
\begin{theo}\label{ppp} Pour une place inerte non ramifi\'ee $v$, et pour tout
$\epsilon>0$.

\item On a:
$$ \dim H_{\mathfrak{q},\theta}^R(X^0(\mathfrak{c}\varpi_{w}^k)) \ll_\epsilon \text{vol}\left(X^0(\mathfrak{c}\varpi_{w}^k)\right)^{\frac{\frac{1}{|e|+1}mn+\alpha(W)+1-|e|}{\alpha(V)}+\epsilon}$$
Et, sous l'hypoth\`ese (\ref{rangf}):
$$\dim H_{\mathfrak{q},\theta}^R(X^0(\mathfrak{c}\varpi_{w}^k))\gg_\epsilon\text{vol}\left(X^0(\mathfrak{c}\varpi_{w}^k)\right)^{\frac{\frac{1}{|e|+1}m(n-2m)+\alpha(W)+1-|e|}{\alpha(V)}-\epsilon}$$
\end{theo}
Par exemple, pour $O(n,1)$ c'est \`a dire pour la g\'eom\'etrie hyperbolique, on
obtient pour le $i^\text{i\`eme}$ nombre de Betti les exposants de
minoration et de majoration:
$$\frac{2i}{n}\left(1-\frac{(2i-1)}{n+1}\right)\text{ et }\frac{2i}{n}\left(1+\frac{(2i+1)}{n+1}\right).$$
Dans le cas $U(n,1)$, c'est \`a
dire pour la g\'eom\'etrie hyperbolique complexe, on obtient pour la
cohomologie holomorphe de degr\'e $i$ les exposants: 
$$\frac{i}{n-1}\left(1-\frac{i}{n-1}\right)\text{ et }\frac{i}{n+1}\left(1+\frac{i}{n+1}\right).$$

\section{Preuve de la majoration locale}\label{partie3}
Nous prouvons le th\'eor\`eme \ref{maj}. Le probl{\`e}me {\'e}tant local non archim{\'e}dien, nous supposons dor{\'e}navant, sans changer les notations, que $F,...$ sont des objets locaux et que la caract{\'e}ristique r{\'e}siduelle de $F$ est diff{\'e}rente de $2$. 
\subsection{Op{\'e}rateurs de Hecke}
 Soit $G$ un groupe alg{\'e}brique d{\'e}fini sur
$F$. On note $\mathcal{H}(G)$ les fonctions {\`a} support compact localement
constantes sur $G(F)$. Soit $dg$ une mesure de Haar sur
$G(F)$. L'op{\'e}ration de convolution:
\[\vhi'*\vhi(g)=\int_G\vhi'(gh^{-1})\vhi(h)dh\]
munit $\mathcal{H}(G)$ d'une structure
d'alg{\`e}bre {\`a} idempotents, on appelle cette alg\`ebre l'alg\`ebre de
Hecke de $G$. Soit $K$ un sous-groupe compact ouvert de $G(F)$, on note $\mathcal{H}(G,K)$ les
{\'e}l{\'e}ments de $\mathcal{H}(G)$ invariants {\`a} droite et {\`a} gauche par
$K$. L'alg\`ebre $\mathcal{H}(G,K)$ admet comme \'el\'ement neutre la fonction
$e_{K}$ qui est {\`a} une constante pr{\`e}s (d{\'e}pendante de la mesure de Haar) la
fonction caract{\'e}ristique de $K$ (on a donc $\mathcal{H}(G,K)=e_K\mathcal{H}(G)e_K$). Soit $(\pi,V_\pi)$ une repr{\'e}sentation irr{\'e}ductible lisse de $G(F)$, alors $\mathcal{H}(G)$ agit sur $\pi$ par la formule:
$$v\mapsto \pi(\vhi)(v)=\int_{G(F)}\vhi(g)\pi(g)v.$$
Nous notons $V_\pi^K$ les vecteurs de $V_\pi$ fixe sous l'action de $K$. L'idempotent $e_K$ de $\mathcal{H}$ est le projecteur sur l'espace $V_\pi^K$. En particulier $\mathcal{H}(G,K)$ agit sur $V_\pi^K$. Enfin, $\mathcal{H}(G)$ est muni d'une involution:
$$f\mapsto \widehat{f}(g)=f(g^{-1})$$
qui v\'erifie 
\begin{eqnarray}\label{convolution}\widehat{f*f'}=\widehat{f'}*\widehat{f}\end{eqnarray}
et pr{\'e}serve les espaces $\mathcal{H}(G,K)$. Nous noterons $\mathcal{H}_V$ (resp. $\mathcal{H}_W$) l'alg{\`e}bre de Hecke du groupe $U(V)$ (resp. $U(W)$). Nous allons calculer l'action des op{\'e}rateurs de Hecke sur les fonctions theta.
Soit $S$ un mod{\`e}le lisse de la repr{\'e}sentation de Weil de $\text{Sp}(\mathbf{W}_v)$. Rappelons que l'image de $\pi$ par la correspondance theta est l'unique quotient irr{\'e}ductible du module suivant:
$$\Theta(\pi,V)=(\pi\otimes S)_{U(W)}$$
o{\`u} la notation en indice d{\'e}signe les coinvariants sous l'action de $U(W)$. Nous notons $p$ la projection de $\pi\otimes S$ dans $\Theta(\pi,V)$. L'action des op{\'e}rateurs de Hecke sur les fonctions theta v{\'e}rifie
les relations suivantes:
\begin{lemm}\label{action0} Soit $h\in \mathcal{H}_W$, $h'\in \mathcal{H}_V$, $f\in\pi$ et $\vhi\in S$
\begin{itemize}
\item $p(\pi(h)f\otimes\vhi)=p(f\otimes\om(\widehat{h})\vhi)$
\item $h'\bullet p(f\otimes\vhi)=p(f\otimes \om(h')\vhi)$
\end{itemize}
\end{lemm}
Le premier point vient de la d{\'e}finition des coinvariants, le
deuxi{\`e}me de l'{\'e}quivariance de la projection $p$ sous l'action de $U(V)$. 
\subsection{Premi\`ere majoration}
Soit $K_V$ un compact de $U(V)$.
\begin{prop}\label{smaj}
Supposons qu'il existe un sous-espace de dimension finie $S_0\subset S$ tel que:
$$S^{1\times K_V}=\om_\chi(\mathcal{H}_W\times 1)S_0$$
Soit $K_W$ un compact de $U(W)$ qui fixe les {\'e}l{\'e}ments de $S_0$ (cela existe toujours). Alors:
$$\dim\Theta(\pi,V)^{K_V}\leq \dim S_0\times\dim\pi^{K_W}$$
\end{prop}
\begin{proof}
En appliquant l'idempotent $ e_{K_V}$ {\`a} la surjection $U(V)$-{\'e}quivariante $p$ on obtient que:
$$\pi\otimes S^{1\times K_V}\twoheadrightarrow \Theta(\pi,V)^{K_V}$$
(le symbole ''$\twoheadrightarrow$'' voulant dire se surjecte). L'hypoth{\`e}se du lemme montre que: 
$$\pi\otimes \om_\chi(\mathcal{H}_W\times 1)S_0\twoheadrightarrow \Theta(\pi,V)^{K_V}$$
Comme les vecteurs de $S_0$ sont fixes sous l'action de $K_W$ on a:
$$\om_\chi(\mathcal{H}_W\times 1)S_0=\om_\chi(\mathcal{H}_We_{K_W}\times 1)S_0$$
Le premier point du lemme \ref{action0}, l'\'egalit\'e (\ref{convolution}) et le fait que $\widehat{e_{K_W}}=e_{K_W}$ montrent que 
\begin{eqnarray*}
p(V_\pi\otimes \om_\chi(\mathcal{H}_We_{K_W}\times 1)S_0)&=&p(\pi(\widehat{e_{K_W}}\mathcal{H}_W)V_\pi\otimes S_0)\\
&=& p(V_\pi^{K_W}\otimes S_0)
\end{eqnarray*}
On a donc une surjection:
$$V_\pi^{K_W}\otimes S_0\twoheadrightarrow \Theta(\pi,V)^{K_V}.$$
\end{proof}
\subsection{R{\'e}sultats de Waldspurger}
Rappelons que la repr{\'e}sentation de Weil avant d'{\^e}tre une repr{\'e}sentation de $\text{Mp}(\mathbf{W})(\A)$ est l'unique repr{\'e}sentation irr{\'e}ductible de caract{\`e}re central $\psi$ des points ad{\'e}liques du groupe de Heisenberg de $\mathbf{W}$:
$$H(\mathbf{W})=\{(w,t)\ |w\in\mathbf{W}\text{ et }t\in F\}$$
muni de la multiplication $(w,t)(w',t')=(w+w',t+t'+\frac{1}{2}\langle w,w'\rangle)$.
On note $\mathfrak{f}_\psi$ le conducteur de $\psi$, notons $\mathbf{f}$ le
corps r\'esiduel de $F$. Rappelons que la notion de dualit{\'e} dans les
r{\'e}seaux d{\'e}pend d'un choix d'id{\'e}al. Nous avons d{\'e}fini
$\mathfrak{f}$ (resp. $\mathfrak{f}_\psi$) comme le conducteur de
$\psi\circ\tr_E$ (resp. $\psi$) et choisi deux id{\'e}aux $\mathfrak{f}_V$
et $\mathfrak{f}_W$ v{\'e}rifiant
$\mathfrak{f}_V\mathfrak{f}_W=\mathfrak{f}$. L'espace vectoriel $\mathbf{W}=W\otimes_E V$
(resp. $W$,$V$) est muni d'une forme symplectique
(resp. $\pm\eta$-hermitienne) et nous consid{\'e}rons la dualit{\'e}
par rapport {\`a} $\mathfrak{f}_\psi$
(resp. $\mathfrak{f}_V$,$\mathfrak{f}_W$). On v{\'e}rifie que si $L_1$
(resp. $L_2$) est un r{\'e}seau de $W$ (resp. $V$), on a $(L_1\otimes_\mathcal{O}
L_2)^*=L_1^*\otimes_\mathcal{O} L_2^*$. Rappelons que nous avons {\'e}galement fix{\'e} des r{\'e}seaux presque autoduaux $L_V$ et $L_W$. Supposons dans un premier temps que $L_V$ et $L_W$ sont autoduaux alors $L=L_W\otimes_{\mathcal{O}} L_V$ l'est aussi. Soit $H(L)\subset H(\mathbf{W})$ les {\'e}l{\'e}ments de la forme $(l,t)$ avec $l\in L$ l'application $\psi_L(l,t)=\psi(t)$ est un caract{\`e}re de $H(L)$ et la repr{\'e}sentation:
$$\text{ind}_{H(L)}^{H(\mathbf{W})}\psi_L$$ 
est un mod{\`e}le de la repr{\'e}sentation de Weil appel{\'e} mod{\`e}le latticiel. Cet espace est isomorphe {\`a}:
$$S(\mathbf{W}//L)=\{f\in S(\mathbf{W})|\ f(w+l)=\psi(-\frac{1}{2}B(l,w))f(w)\}$$
muni de l'action du groupe de Heisenberg donn{\'e}e par:
\begin{eqnarray}\label{rho}\rho(w)f(w')=\psi(\frac{1}{2}B(w',w))f(w+w')\end{eqnarray} 
L'int{\'e}r{\^e}t de ce mod{\`e}le par rapport {\`a} un mod{\`e}le de Schr{\"o}dinger est que l'action du compact maximal $K(L_V)$ de $U(V)$ est tr{\`e}s simple et donn{\'e}e par:
\begin{eqnarray}\label{K}\om_\chi(k)\vhi(x)=\lambda_\chi(k)\vhi(k^{-1}x).\end{eqnarray}
o\`u $\lambda_\chi$ est un caract\`ere de $K(L_V)$ calcul\'e par Pan dans
\cite{Pan1} (cf. partie \ref{scindage}).
On note $S[r]$ les fonctions de $S(\mathbf{W}//L)$ {\`a} support dans
$\frac{1}{\varpi^{-r}}L$. Waldspurger d{\'e}montre dans \cite{MVW} sous l'hypoth{\`e}se suppl{\'e}mentaire que $E/F$ est non ramifi{\'e} que:
\begin{theo}\label{thmW1} On a si $k$ est grand que le conducteur de $\chi$:
$$
\begin{array}{lll}
S^{1\times K(\varpi^k)}&=&\om(\mathcal{H}_W\times 1)S[\frac{k}{2}]\text{ si $k$ pair}\\
S^{1\times K(\varpi^k)}&=&\om(\mathcal{H}_W\times 1)S[\frac{k+1}{2}]^{1\times K(\varpi^k)}\text{ si $k$ impair}
\end{array}
$$
\end{theo}
En utilisant le fait que si $k$ est pair $S[k]$ est fixe par l'action de
$K_W(\varpi^{2k})$, la proposition \ref{smaj} donne le th\'eor\`eme
\ref{maj} dans le cas non ramifi\'e. D\'emontrons ce fait. Soit $s\in S[k]$, le support de $s$ est bien
stable par $K_W(\varpi^{2k})$. Il suffit donc de prouver que si $w\in
\frac{1}{\varpi^k}L$, on a bien $\om(k)s(w)=s(w)$. Comme
$k^{-1}w=(k^{-1}-1)w+w$ avec le premier terme de la somme qui appartient \`a
$L$, on a d'apr\`es (\ref{rho}):
\begin{eqnarray*}
\om(k)s(w)=\psi(\frac{1}{2}B((k^{-1}-1)w,w)s(w)
\end{eqnarray*}La condition $k\in K_W(\varpi^{2k})$ est donc la condition qui assure que
$B((k^{-1}-1)w,w)\in \mathfrak{f}_\psi$. En g{\'e}n{\'e}ral quand $L_V$ et $L_W$ ne sont pas autoduaux, on introduit:
$$A=L_W^*\otimes L_V\cap L_W\otimes L_V^*\text{ et }B=L_W^*\otimes L_V+L_W\otimes L_V^*$$
On a $B=A^*$ et $\varpi B\subset A\subset B$. On pose $\mathbf{b}=B/A$ qui est muni naturellement par r{\'e}duction d'une forme symplectique non d{\'e}g{\'e}n{\'e}r{\'e}e sur $\mathbf{f}$. On choisit (c'est possible) un r{\'e}seau autodual $L$ de $\mathbf{W}$ qui v{\'e}rifie les inclusions:
$$A\subset L\subset B$$ 
On remarque que cela revient {\`a} choisir un sous-espace isotrope maximal de $\mathbf{b}$. Le probl{\`e}me du mod{\`e}le latticiel de la partie pr{\'e}c{\'e}dente est que le stabilisateur de $L$ ne contient pas $K(L_V)\times K(L_W)$, on change donc le mod{\`e}le pour avoir {\`a} nouveau une formule explicite de l'action de $K(L_V)\times K(L_W)$. On a:
$$\text{ind}_{H(L)}^{H(\mathbf{W})}\psi_L=\text{ind}_{H(B)}^{H(\mathbf{W})}\text{ind}_{H(F)}^{H(B)}\psi_L$$
L'induite $\text{ind}_{H(F)}^{H(B)}\psi_L$ est un mod{\`e}le de la repr{\'e}sentation du groupe de Heisenberg fini $H(\mathbf{b})$ que l'on notera $(\rho,\mathbb{S})$. Le mod{\`e}le est alors {\'e}gal {\`a}:
$$S(\mathbf{W},\mathbb{S})=\{f:\mathbf{W}\rightarrow \mathbb{S}|\ f(w+b)=\psi(-\frac{1}{2}B( b,w))\rho(b)f(w)\}$$
L'action de $H(\mathbf{W})$ est toujours donn{\'e} par la formule (\ref{rho}). L'avantage de ce mod{\`e}le est que la repr{\'e}sentation de Weil de $K(L_V)$ est alors explicitement donn{\'e}e par la formule si $k\in K(L_V)$:
$$\om_\chi(k)\vhi(x)=\lambda_\chi(k)\rho(k)\vhi(k^{-1}x)$$
o{\`u} $\rho(k)$ d{\'e}signe la repr{\'e}sentation de Weil de
$\text{Sp}(\mathbf{b})$. On note pour un r\'eseau $L'\subset \mathbf{W}$
$S[L']$ les fonctions de $S(\mathbf{W},\mathbb{S})$ \`a support dans
$L$. 
\paragraph*{Remarque}
On dira plus g\'en\'eralement qu'un r\'eseau est bon si:
$$\varpi L^*\subset L\subset L^*$$
La construction pr\'ec\'edente s'\'etend aussi si on suppose simplement que $L_V$
et $L_W$ sont bons plut\^ot que presque autoduaux. Nous nous servirons de ce
fait dans la partie \ref{pararam}.

Waldspurger d{\'e}montre alors dans l'article \cite{Wald} le r\'esultat
suivant:
\begin{theo} On a pour $k$ suffisamment grand:
\begin{itemize}
\item Si k pair,
$$S^{1\times K(\varpi^k)}=\om(\mathcal{H}_W\times 1)S[\varpi^{-\frac{k}{2}}B]^{1\times K(\varpi^k)}$$
\item Si k impair,
$$S^{1\times K(\varpi^k)}=\om(\mathcal{H}_W\times 1)S[\varpi^{-\frac{k+1}{2}}A]^{1\times K(\varpi^k)}$$
\end{itemize}
\end{theo}
$S[\varpi^{-k}B]$ est fixe par
$K_W(\varpi^{2k+1})\times 1$ et $S[\varpi^{-k}A]$ est fixe par
$K_W(\varpi^{2k})\times 1$. Ainsi en utilisant \`a nouveau la proposition
\ref{smaj}, on d\'eduit le th\'eor\`eme \ref{thmmin}:
\begin{coro}\label{majR}Pour $k$ suffisamment grand:
 $$\dim\Theta(\pi,V)^{K(\varpi_w^k)}\leq (\mathbb{N}w)^{\frac{k+1}{2}\dim V\dim W}\dim\pi^{K(\varpi_w^{k+1})}$$
\end{coro}
Il est naturel de se poser la question de savoir si on peut d{\'e}montrer qu'une
partie des fonctions de $S[\frac{r}{2}]$ s'injecte dans $\theta(\pi,V)^{K(\varpi_w^r)}$. Nous
prouvons un r{\'e}sultat de ce type dans la partie suivante.

\section{Formule du produit scalaire de Rallis}\label{SWF}
On suppose dans la suite que les dimensions v\'erifient la condition \ref{rangf}.
Nous allons donner sous ces hypoth\`eses, pour une place $v$ impaire la
preuve de la proposition \ref{propmaj} et du th\'eor\`eme \ref{thmmaj}. Nous allons prouver une formule
du produit scalaire de Rallis locale, c'est la proposition \ref{Form}. Sous la conjecture 1.2 de \cite{KR} on peut d\'emontrer cette formule de mani\`ere locale seulement sous l'hypoth\`ese de rang stable. On prouve ensuite la formule de mani\`ere globale sous la condition \ref{rangf}.
\subsection{Le groupe doubl{\'e}}\label{class}
Nous notons $-W$ l'espace hermitien muni de la forme sesquilin{\'e}aire $-\langle,\rangle$ et $2W=W\oplus -W$. Nous notons $\Delta^+(W)$ (resp. $\Delta^-(W)$) les sous-espaces vectoriels de $2W$ form{\'e}s par les {\'e}l{\'e}ments de la forme $(w,w)$ (resp. $(w,-w)$). Ce sont deux sous-espaces isotropes maximaux de $2W$ en dualit{\'e}. On note $P_\Delta$ le stabilisateur de $\Delta^+(W)$, c'est le parabolique de Siegel de $U(2W)$. On note $\iota$ le plongement canonique: 
$$U(W)\times U(-W)\rightarrow U(2W)$$
Le quotient $P_\Delta\bk U(2W)$ s'identifie {\`a} l'ensemble des sous-espaces isotropes maximaux de $2W$ que l'on note $\Omega(2W)$, on s'int{\'e}resse aux orbites de $U(W)\times U(-W)$ dans $\Omega(2W)$. Il est connu \cite{GPSR} que l'orbite d'un espace isotrope maximal $Z$ est d{\'e}termin{\'e} par l'invariant:
$$k=\dim\left(Z\cap( W\oplus 0)\right)=\dim\left(Z\cap( 0\oplus -W)\right)$$
Nous notons $\Omega_k$ l'orbite correspondante. Un repr{\'e}sentant est
donn{\'e} par $2X_k\oplus \Delta^+(W_k)$. Le stabilisateur, que nous
noterons $\text{St}_k$, de $2X_k\oplus\Delta^+(X_k)$ est contenu dans $P_k\times P_k$ et vaut:
$$\text{St}_k=\text{GL}(X_k)\times \text{GL}(X_k)\times\Delta(U(W_k))\rtimes N_k\times N_k$$
($\Delta$ d\'esigne la diagonale d'un produit).

\subsection{Int{\'e}grale z{e}ta locale}
Soit $\widetilde{\chi}$ un caract{\`e}re unitaire de $E^*$. On consi\-d{\`e}re une section de l'induite:
$$\phi(\bullet,s)\in I_{P_\Delta}^{U(2W)}|\bullet|^s\widetilde{\chi}$$
On peut alors d{\'e}finir pour $v,v'\in V_\pi$ la fonction z{e}ta locale:
$$Z(s,\phi,v,v')=\int_{U(W)}\phi(\iota(g,1),s)(\pi(g)v,v')dg.$$
Vu que $\pi$ et $\widetilde{\chi}$ sont unitaires, cette int{\'e}grale converge absolument pour $\Re(s)\geq \frac{m+e}{2}$, elle admet de plus un prolongement m\'eromorphe. Soit $s_0\in\C$ tel que $\Re(s_0)\geq \frac{m+e}{2}$ alors $Z(s_0)$ d{\'e}finit un {\'e}l{\'e}ment non trivial de:
\begin{eqnarray}\label{espace}
\hom_{U(W)\times U(W)}(I_{P_\Delta}^{U(2W)}\widetilde{\chi}|\bullet|^{s_0},\pi\otimes\widetilde{\chi}\pi^\vee)
\end{eqnarray}
En suivant Kudla et Rallis \cite{KR}, on peut conjecturer que:
\begin{conj}\label{cor1} On a:
\begin{eqnarray}\label{multi1}
\dim\hom_{U(W)\times U(W)}(I_{P_\Delta}^{U(2W)}\widetilde{\chi}|\bullet|^{s_0},\pi\otimes\widetilde{\chi}\pi^\vee)=1.
\end{eqnarray}
\end{conj}

\subsection{La m{\'e}thode du double}
On d\'efinit le caract\`ere $\widetilde{\chi}$  par la formule suivante:
\begin{itemize}
\item Cas 1 $\eta=1$ $\widetilde{\chi}=\chi_V$
\item Cas 1 $\eta=-1$ $\widetilde{\chi}=1$
\item Cas 2 $\widetilde{\chi}={\chi}^n$
\end{itemize}

 Soit $S$ un mod{\`e}le de la repr{\'e}sentation de Weil de $H(\mathbf{W})$. Vu que: $$H(2\mathbf{W})=H(\mathbf{W})\times H(-\mathbf{W})/\{(t,-t)|t\in F\}$$
 $S\otimes S^\vee$ est un mod{\`e}le de la repr{\'e}sentation de Weil de $H(2\mathbf{W})$. Cependant on a introduit pr\'ec\'edemment une polarisation
naturelle $2\mathbf{W}=\Delta^+(W)\oplus \Delta^-(\mathbf{W})$ de
$2\mathbf{W}$. On a donc un mod\`ele de Schr\"odinger naturel de la
repr\'esentation de Weil de $H(2\mathbf{W})$, donn\'e par les fonctions de
Schwarz sur $\Delta^-(\mathbf{W})=\mathbf{W}$. L'isomorphisme entre les repr\'esentaions de Weil de $H(\mathbf{W})$ est donn\'e par l'op\'erateur d'entrelacement $C_V$:
\begin{eqnarray}\label{iso}
C_V=\left\{\begin{array}{ccc}
S\otimes S^\vee&\rightarrow& S(W\otimes V)\\
s\otimes s^\vee&\mapsto& (w\otimes v\mapsto (\rho(2w\otimes v)s,s^\vee).
\end{array}\right.
\end{eqnarray}
La repr\'esentation de Weil $\om_\chi$ de $U(2W)$ sur $S(V\otimes W)$ (c'est \`a dire le choix de cocycle) est d\'efini par Kudla dans \cite{K} et est caract\'eris\'e par le fait que le Levy $P_\Delta$ de $U(2W)$ agit par la formule suivante sur $\vhi\in S(W\otimes V)$:
\begin{eqnarray}\label{para}
\om_\chi(a)\vhi(x)=\widetilde{\chi}(\det a)|\det a|^{\frac{n}{2}}\vhi(a^{-1}x).
\end{eqnarray}
On d\'emontre dans la partie \ref{scindage} (lemme \ref{lemRes}) qu'en restriction \`a $U(W)\times U(W)$ l'op\'erateur d'entrelacement $C_V$ induit un isomorphisme:
\begin{eqnarray}\label{isotor}
\om_\chi\otimes\widetilde{\chi}\om_\chi^\vee={\om_\chi}_{|U(V)\times U(V)}.
\end{eqnarray}
 Posons $s_0=\frac{n-(m+e)}{2}$, d'apr\`es (\ref{para}) l'application suivante que l'on notera $I_V$ d\'efini un op\'erateur d'entrelacement:
\begin{eqnarray}\label{section}
\begin{array}{ccc}
S(W\otimes V)&\rightarrow& I_{P_\Delta}\widetilde{\chi}|\bullet|^{s_0}\\
\vhi&\mapsto& (h\mapsto \om_\chi(h)\vhi(0))
\end{array}
\end{eqnarray}
L'image de cette application est not{\'e}e $R(\widetilde{\chi},V)$, un th{\'e}or{\`e}me de Kudla, Rallis et Sweet  montre que cet espace est exactement $\Theta(V,\mathbf{1})$. Dans le rang stable, le r{\'e}sultat suivant est d{\'e}montr{\'e} par Kudla, Rallis et Sweet:
\begin{prop}\label{surj1}
L'application $I_V$ est surjective et $R(\widetilde{\chi},V)=I_{P_\Delta}\widetilde{\chi}|\bullet|^{s_0}$ est irr{\'e}ductible.
\end{prop}
En restriction {\`a} $U(W)\times U(-W)$ on obtient un op{\'e}rateur d'entrelacement surjectif $I_V$:
$$
S\otimes\widetilde{\chi} S^\vee\rightarrow I_{P_\Delta}\widetilde{\chi}|\bullet|^{s_0}
$$
v{\'e}rifiant d'apr\`es la formule (\ref{iso}) pour tout $(g,g')\in U(W)\times U(-W)$:
\begin{eqnarray}\label{ent}
I_V(s\otimes\overline{s'})(\iota(g,g'))=\widetilde{\chi}(\det g')(\om_\chi(g)s,\om_\chi(g')s').
\end{eqnarray}
La proposition \ref{surj1} permet de construire un deuxi{\`e}me {\'e}l{\'e}ment de l'espace d'entrelacement d\'efini par l'\'egalit\'e (\ref{espace}). En effet, 
consid{\'e}rons la projection $p:S\otimes\pi\rightarrow \theta(\pi)$, comme les repr{\'e}sentations sont toutes unitaires on a {\'e}galement une application duale $p^\vee: S^\vee\otimes\pi^\vee\rightarrow \theta(\pi)^\vee$. Introduisons une nouvelle notation, pour une repr{\'e}sentation $\pi$ nous notons $C:\pi\otimes\pi^\vee\rightarrow \C$ l'application coefficient. Consid{\'e}rons :
$$C\circ(p\otimes p^\vee):S\otimes\widetilde{\chi} S^\vee\otimes\pi\otimes\widetilde{\chi}\pi^\vee\rightarrow\C,$$
cette application se factorise en un {\'e}lement de:
$$\hom(R(\widetilde{\chi},V),\pi\otimes\widetilde{\chi}\pi^\vee).$$
L'addition de la proposition \ref{surj1} et de la conjecture \ref{cor1}
(puisque $s_0>\frac{m+r+e}{2}$) montrerait qu'il existe une constante non nulle $c$ telle que:
\begin{prop}[Formule du produit scalaire de Rallis locale]\label{Form} Sous l'hypoth\`ese \ref{rangf}
\[\left(p(\vhi_1\otimes v_1),p(\vhi_2\otimes v_2)\right)=c\int_{U(W)}(\om_\chi(g)\vhi_1,\vhi_2)(\pi(g)v_1,v_2)dg.\]
\end{prop}
Nous donnons une preuve de cette proposition dans la partie suivante.
\subsection{Formule globale}
On se donne deux formes automorphes cuspidales $\pi$ et $\pi'$ de $U(W)$. Soit
$f$ (resp. $f'$) un vecteur de $\pi$ (resp. $\pi'$). Soit
$\vhi_1$ et $\vhi_2$ deux vecteurs de $S$. On cherche \`a calculer
\begin{eqnarray}\label{prodsca}
(\theta(f,\vhi_1),\theta(f',\vhi_2)).
\end{eqnarray}
Pour simplifier les notations nous noterons $[G]$ le quotient $G(F)\bk
G(\A)$ pour un groupe alg\'ebrique $G$ d\'efini sur $F$. On a en
d\'eveloppant que le produit scalaire (\ref{prodsca}) vaut:
\begin{eqnarray*}
\int_{[U(W)]\times [U(W)]}f(h)\overline{f'(h')}\int_{[U(V)]}\Theta(g,h,\vhi_1)\overline{\Theta(g,h',\vhi_2)}dgdhdh'
\end{eqnarray*}
Les formules (\ref{iso}), (\ref{isotor}) et (\ref{section}) montrent que:
\begin{eqnarray*}
\Theta(g,h,\vhi_1)\overline{\Theta(g,h',\vhi_2)}=\Theta\left(g,(h,h'),C_V(\vhi_1\otimes\overline{\vhi_2})\right)\widetilde{\chi}^{-1}(\det
h')
\end{eqnarray*}
On a donc que (\ref{prodsca}) vaut:
\begin{eqnarray*}
\int_{[U(W)]\times
[U(W)]}f(h)\overline{f'(h')}\widetilde{\chi}^{-1}(\det h')\int_{[U(V)]}\theta(g,(h,h'),C_V(\vhi_1\otimes\overline{\vhi_2})dgdhdh'
\end{eqnarray*}
La derni\`ere int\'egrale est convergente et d\'efinie une fonction \`a croissance mod\'er\'ee
d\`es que le crit\`ere de Weil (6) est v\'erifi\'e \cite{Weil}. Comme les
repr\'esentations $\pi$ et $\pi'$ sont cuspidales, on en d\'eduit le dernier
point du th\'eor\`eme \ref{conv}. Rappelons que l'on a pos\'e
$s_0=\frac{n-(m+e)}{2}$. L'\'el\'ement:
$$\phi=I_V\circ C_V(\vhi_1\otimes\overline{\vhi_2})\in I_{P_\Delta}^{U(2W)}\widetilde{\chi}|\bullet|^{s_0}.$$
On le prolonge en une famille de section: 
$$\phi(s)\in
I_{P_\Delta}^{U(2W)}\widetilde{\chi}|\bullet|^{s}$$
 telle que $\phi(s_0)=\phi$. On peut alors former la s\'erie d'Eisenstein:
\begin{eqnarray}\label{defE}
E(g,s,\phi)=\sum_{\gamma\in P_\Delta(F)\bk U(2W)(F)}\phi(\gamma g,s)\ (g\in U(2W)(\mathbb{A}))
\end{eqnarray}
cette s\'erie est convergente pour $\Re(s)\gg 0$ et admet un prolongement
m\'eromorphe. On conjecture le fait suivant:
\begin{conj} Supposons les conditions de Weil v\'erifi\'ee alors:
\begin{itemize}
\item La s\'erie d'Eisenstein pr\'ec\'edente est holomorphe au voisinage de $0$.
\item Il existe une constante non nulle $c$ telle que:
$$E(g,s_0,\vhi)=c\int_{[U(V)]}\Theta(g,h,\vhi_1\otimes\overline{\vhi_2})dh$$
\end{itemize}
\end{conj}
Cette conjecture a \'et\'e d\'emontr\'ee par Kudla et Rallis pour la paire
orthogonale-symplectique et par Ichino pour les paires de groupes
unitaires. Dans tout les cas, sous l'hypoth\`ese (\ref{rangf}) le r\'esultat est vrai (\cite{Weil}). On obtient finalement que:
$(\theta(f,\vhi_1),\theta(f',\vhi_2))$ est un multiple non nul de la valeur en $s_0$ de
l'int\'egrale zeta:
\begin{eqnarray*}
Z(s,\phi,f,f')&=&\int_{[U(W)]\times[U(W)]}f(h)\overline{f'(h')}\widetilde{\chi}^{-1}(\det h')E((h,h'),s,)dhdh'
\end{eqnarray*}
La classification des $U(W)\times
U(W)$-orbites de $P_\Delta\bk U(2W)$ (cf. paragraphe \ref{class}) et le fait que $f$ et $f'$ soient
cuspidales \cite{GPSR} montre que cette int\'egrale vaut:
\begin{eqnarray*}
&&\int_{\Delta(U(W)(F))\bk U(W)(\A)\times U(W)(\A)}f(h)\overline{f'(h')}\widetilde{\chi}^{-1}(\det h')\phi(h,h'),s,)dhdh'\\
&=&\int_{\Delta(U(W)(F))\bk U(W)(\A)\times
U(W)(\A)}f(h)\overline{f'(h')}\widetilde{\chi}^{-1}(\det
h')\phi(h'^{-1}h,1),s,)dhdh'
\\
&=&\int_{U(W)(\A)}\int_{[U(W)]}f(h)\overline{f'(hu^{-1})}dh\phi((u,1),s)du
\end{eqnarray*}
La deuxi\`eme int\'egrale est le produit scalaire de $f$ et $u^{-1}\bullet f$,
elle est donc nulle si $\pi\neq\pi'$. On a donc bien le lemme
\ref{injW}. Supposons que $f=\otimes f_v$, $f'=\otimes f'_v$ dans
$\pi=\otimes\pi_v$ et $\phi(s)=\otimes \phi_v(s)$ alors on obtient la
formule de Rallis:
\begin{eqnarray*}\label{FPR}
(\theta(f,\vhi_1),\theta(f',\vhi_2))&=&\text{valeur en $s_0$ de: } \prod_v \int_{U(W)_v}(\pi(h)f_v,f_v')\phi((g,1),s)dg
\end{eqnarray*}
Comme $\phi((g,1),s_0)=(\om_\chi(g)\vhi_1,\vhi_2)$ cette formule sugg\`ere
bien la proposition \ref{Form}.

Nous pouvons maintenant prouver la formule locale en une place $v$ en toute g\'en\'eralit\'e. Il suffit pour cela d'utiliser la factorisation du th\'eor\`eme \ref{conv} et la formule pr\'ec\'edente (\ref{FPR}) en fixant des tenseurs purs en toutes places diff\'erentes de $v$. La proposition $4.3$ est donc vraie.

\section{Minoration}\label{min}

\subsection{Construction de fonctions de Schwarz}
On consid{\`e}re dans cette partie la paire duale $U(W)\times U(V_0)$
o{\`u} $U(V_0)$ est le groupe d{\'e}ploy{\'e} de rang $m$ c'est {\`a} dire
que $V_0$ admet une d{\'e}composition de Witt $V_0=X_0\oplus Y_0$ avec
$\dim X_0=m$. Soit $\mathbf{W}_0=W\otimes V_0$ et soit $S_0$ un mod\`ele de la
repr\'esentation de $H(\mathbf{W}_0)$. Pour tout $k$ suffisamment grand, on construit un {\'e}l{\'e}ment $\vhi_{k}\in S_0$, v{\'e}rifiant pour tout $g\in U(W)$:
$$(\om(g)\vhi_k,\vhi_k)=1_{K(\varpi^k)}(g)$$
Donnons d'abord une id\'ee de la construction. On peut consid{\'e}rer que la d\'ecomposition de Witt de $V_0$ est $V_0=W\oplus W^*$, un mod{\`e}le de Schr{\"o}dinger est alors donn{\'e} par $S(W\otimes W^*)=S(\text{End}(W))$. Le groupe $U(W)$ est inclu dans le stabilisateur de $W\otimes W^*$, l'action de $U(W)$ sur $S(\text{End(W)})$ est donc lin\'eaire. C'est \`a dire qu'{\`a} un caract{\`e}re pr{\`e}s l'action de $U(W)$ est donn{\'e} par:
$$\om_\chi(g)\vhi(x)=\vhi(g^{-1}\circ x).$$
L'{\'e}l{\'e}ment:
$$\vhi_k(x)=1_{\text{id}_W+\varpi^k\text{End} L_W}$$ 
convient. Une comparaison avec le mod{\`e}le latticiel $S_0$ de la
repr\'esentation de Weil (ou une construction directe) montre qu'en dehors
des places de ramification l'{\'e}l{\'e}ment $\vhi_k\in S_0[k]$. Plus pr\'ecis\'ement, on a un isomorphisme d'espace h\'ermitien $V_0=W\oplus
W$, avec $V_0$ muni de la forme:
$$(w_1+w_1',w_2+w_2')=\langle w_1,w_2'\rangle-\langle w_1',w_2\rangle$$
Un r\'eseau autodual naturel de $V_0$ est donn\'e par:
$$L_{V_0}=L_W^*\oplus\varpi^{-r}L_W$$
(o\`u $r$ la valuation du conducteur de $\psi\circ\tr$). Un bon r\'eseau du produit tensoriel est donn\'e par:
$$A=L_W\otimes L_{V_0}.$$
Un r\'eseau autodual $L$ v\'erifiant les inclusions:
$$A\subset L\subset A^*$$
est donn\'e par:
$$L=L_W\otimes L_W^*\oplus \varpi^{-r}L_W^*\otimes L_W.$$
On a trois mod\`eles possibles pour la repr\'esentation de Weil, le mod\`ele
de Schr\"odinger $S(W\otimes W)$, le mod\`ele latticiel $S(\mathbf{W}_0//L)$ et
le mod\`ele latticiel g\'en\'eralis\'e
$S(\mathbf{W}_0,\mathbb{S})$. L'isomorphisme entre les deux premiers
mod\`eles est donn\'e par:
$$
\left\{
\begin{array}{ccc}
S(W\otimes W)&\simeq& S(\mathbf{W}_0//L)\\
\vhi&\mapsto& f
\end{array}
\right.
$$
avec:
$$f(w+w')=\int_{L_W\otimes L_W*}\psi(B(w'',w'))\psi(\frac{1}{2}B(w,w'))f(w+w'')dw''$$
On a une identification de $W\otimes W$ avec $\text{End(W)}$ par la formule
suivante:
$$
\left\{
\begin{array}{ccc}
W\otimes W&\simeq& \text{End}(W)\\
w_1\otimes w_2&\mapsto&(w\mapsto \langle w,w_2\rangle w_1)
\end{array}
\right.
$$
On v\'erifie alors que par ces isomorphismes explicites la fonction $\vhi_k$
est \`a support dans $\varpi^{-(k+r_w)}L$ (ce r\'esultat est valide en
caract\'eristique r\'esiduelle $2$). Nous supposons maintenant que toutes les donn\'ees sont non ramifi\'ees, et que $k$ est pair, alors d'apr\`es le th\'eor\`eme \ref{thmW1} et le fait que $\vhi_k\in S^{K_W(\varpi^k)}$, il existe un \'el\'ement $h_k\in \mathcal{H}_{V_0}$ et une fonction $\vhi_k^0\in S[\frac{k}{2}]$ telle que:
\begin{eqnarray}\label{vhik0}
\vhi_k=\om_\chi(h_k)\vhi_k^0.
\end{eqnarray}
Nous posons pour la suite $\vhi_k^1=\widehat{h_k}\vhi_k$.
\subsection{Conclusion}
On a un r{\'e}seau autodual $L_V$ donn{\'e} par la th{\'e}orie globale. La th{\'e}orie des r{\'e}seaux autoduaux montre qu'il existe une d{\'e}composition de Witt de $V$, $V=X\oplus V_1\oplus Y$ avec $\dim X=m$ et $L_V=X\cap L_V\oplus V_1\cap L_V\oplus Y\cap L_V$. Posons $V_0=X\oplus Y$. $V$ se d\'ecompose orthogonalement en $V_0\oplus V_1$. Soit $S_0$ (resp. $S_1$) un mod{\`e}le de la repr{\'e}sentation de Weil de $H(W\otimes V_0)$ (resp. $H(W\otimes V_1)$). Alors $S_0\otimes S_1$ est un mod{\`e}le de la repr{\'e}sentation de Weil de $H(W\otimes V)$, de plus l'action de $U(W)$ sur $S$ est donn{\'e}e par l'action diagonale sur $S_0\otimes S_1$. On a donc si $v,v'\in V_\pi^{K_W(\varpi^k)}$ et $\vhi,\vhi'\in S_1^{K_W(\varpi^k)}$:
\begin{eqnarray*}
&&(p((\vhi_{k}\otimes\vhi)\otimes v),p((\vhi_{k}\otimes\vhi')\otimes v')\\
&=&c\int_{U(W)}(\om(g)\vhi_{k},\vhi_{k})(\om(g)\vhi,\vhi')(\pi(g)v,v')dg\\
&=& (\vhi,\vhi')(v,v')
\end{eqnarray*}
On en d\'eduit donc que l'application:
$$S_1^{K_W(\varpi^k)}\otimes \pi^{K_W(\varpi^k)}\rightarrow \theta(\pi)$$
est injective. Ainsi pour tout compact $K$ de $U(V)$ on a:
\begin{prop}
$$\dim\{\vhi\in S_1^{K_W(\varpi^k)}| \vhi_{k}\otimes \vhi\in S^{K}\}\times\dim\pi^{K(\varpi^k)}\leq \dim\theta(\pi)^K$$
\end{prop}
La proposition \ref{propmaj} est un corollaire de ce r\'esultat:
\begin{coro}\label{minGro} Il exite une constante $c>0$ telle que pour tout $k$ suffisamment grand:
 $$\dim\theta(\pi,V)^{K(\varpi^{2k+c})}\geq \dim\pi^{K(\varpi^k)}$$
\end{coro}
Il suffit de remarquer que pour $c$ suffisamment grand $S[\varpi^{-k}A]$ est fixe
par $K(\varpi^{2k+c})$ et contient $\vhi_k$, on applique alors la
proposition pr\'ec\'edente.

Pour obtenir le th{\'e}or{\`e}me \ref{maj}, il faut utiliser raffiner la
proposition pr\'ec\'edente en utilisant l'\'el\'ement $\vhi_k^0$ d\'efini en (\ref{vhik0}). Nous calculons maintenant le produit scalaire suivant (avec encore $v,v'\in V_\pi^{K_W(\varpi^k)}$ et $\vhi,\vhi'\in S_1^{K_W(\varpi^k)})$:
\begin{eqnarray*}
&&(p((\vhi_{k}^0\otimes\vhi)\otimes v),p((\vhi_{k}^1\otimes\vhi')\otimes v')\\
&=&c\int_{U(W)}(\om(g)\vhi_{k}^0,\widehat{h_k}\vhi_{k})(\om(g)\vhi,\vhi')(\pi(g)v,v')dg\\
\end{eqnarray*}
Remarquons alors que:
\begin{eqnarray*}
(\om(g)\vhi_{k}^0,\widehat{h_k}\vhi_{k})&=&(h_k\om(g)\vhi_{k}^0,\vhi_{k})\\
&=&(\om(g)h_k\vhi_{k}^0,\vhi_{k})
\end{eqnarray*}
La deuxi\`eme ligne \'etant obtenue en utilisant que l'action de $U(W)$ et de $\mathcal{H}_{V_0}$ commute. On obtient alors en continuant comme la premi\`ere fois que:
\begin{eqnarray}
(p((\vhi_{k}^0\otimes\vhi)\otimes v),p((\vhi_{k}^1\otimes\vhi')\otimes v')&=&(\vhi,\vhi')(v,v')
\end{eqnarray}
On en d\'eduit donc pour tout compact $K$:
$$\dim\{\vhi\in S_1^{K_W(\varpi^k)}| \vhi_{k}^0\otimes \vhi\in S^{K}\}\times\dim\pi^{K(\varpi^k)}\leq \dim\theta(\pi)^K$$
\begin{coro}\label{minfin} On a la minoration suivante:
$$
\dim\theta(\pi,V)^{K(\varpi^k)}\geq \dim\pi^{K(\varpi^{k})}\mathbb{N}w^{\frac{k}{2}(\dim V-2\dim
W)\dim
W}
$$
\end{coro}
On applique le r\'esultat pr\'ec\'edent avec $K=K_V(\varpi^k)$,en utilisant que:
$$\vhi_k^0\otimes S_1[\frac{k}{2}]\subset S^{K(\varpi^k)}.$$

\section{Le probl\`eme du cocycle}\label{scindage}

\subsection{La d\'efinition de la repr\'esentation $\om_{\chi}$}

On se donne $(S,\rho)$ un mod\`ele de la repr\'esentation de Weil du
groupe Heisenberg $H(\mathbf{W})$. Rappelons qu'il existe une unique
repr\'esentation projective $\omega$ de $\text{Sp}(\mathbf{W})$ v\'erifiant:

\begin{equation}
\rho(gw)=\omega(g)^{-1}\rho(w)\omega(g)\label{eq:def}\end{equation}
Le probl\`eme du cocycle est de d\'ecrire explicitement, dans le mod\`ele
$S$ choisi, une repr\'esentation $\om_\chi$ de la paire duale $U(W)\times
U(V)$ v\'erifiant la relation pr\'ec\'edente. Nous rappelons la solution de
ce probl\`eme donn\'e par Kudla dans \cite{K}. Par sym\'etrie il suffit de le
faire pour $U(W)$.

\subsubsection{Construction d\'eploy\'ee}
Consid\'erons une paire duale $U(W)\times U(V)$ tellr que l'espace $W$ est d\'eploy\'e. Alors si on se donne une
polarisation $W=X\oplus Y$, un mod\`ele de la repr\'esentation de Weil
de $U(W)\times U(V)$ est donn\'e par le mod\`ele de Schr\"odinger:\[
S( X\otimes V)\]
c'est \`a dire les fonctions de Schwarz de $X\otimes V$. L'action
du Levi de $W$ d\'etermin\'e par la d\'ecomposition $X\oplus Y$, qui est
isomorphe \`a $\text{GL}(X)$ est alors donn\'ee par:\[
\widetilde{\chi}(\det a)\vhi(a^{-1}x),\]
avec $\widetilde{\chi}$:
\begin{itemize}
\item Cas 1 $\eta=1$ $\widetilde{\chi}=\chi_V$
\item Cas 1 $\eta=-1$ $\widetilde{\chi}=1$
\item Cas 2 $\widetilde{\chi}_{|E^*}=\epsilon_{E/F}^{\dim V}$
\end{itemize}

\subsubsection{Construction g\'en\'erale}

On utilise alors la m\'ethode du double, la suite exacte:\[
1\rightarrow\left\{ (t,-t)|\ t\in F\right\} \rightarrow H(2\mathbf{W})\rightarrow H(\mathbf{W})\times H(-\mathbf{W})\rightarrow1\]

montre que la repr\'esentation $S\otimes S^{\vee}$ est un mod\`ele de
la repr\'esentation de Weil de $H(2\mathbf{W})$. La polarisation: 
\begin{eqnarray*}\label{d\'ecomp}
\mathbf{W}=\Delta^{+}(\mathbf{W})\oplus\Delta^{-}(\mathbf{W})
\end{eqnarray*}
montre que $S(\mathbf{W})$ est un autre mod\`ele, l'isomorphisme explicite entre les deux mod\`eles est donn\'e par l'espace:
\begin{equation}
\left\{ \begin{array}{ccc}
S\otimes S^{\vee} & \rightarrow^{I} & S(\mathbf{W})\\
s\otimes s^{\vee} & \mapsto & \left(w\mapsto(\rho(2w)s,s^{\vee})\right)\end{array}\right.\label{eq:iso}\end{equation}
On note $\widetilde{\om_{\chi}}$ la repr\'esentation de $U(2W)$ d\'efinie
dans la section pr\'ec\'edente. On sait par la relation $(\ref{eq:def})$ que:

\[
\om_{\chi}=\widetilde{\om_{\chi}}_{|U(W)\times1}\]

est une repr\'esentation de $S$. De plus l'isomorphisme explicite $(\ref{eq:iso})$
montre que:

\[
I(\om_{\chi}\otimes\om_{\chi}^{\vee}(g,g)s\otimes s^{\vee})(w)=I(s\otimes s^{\vee})(g^{-1}w)\]

Or l'\'el\'ement $(g,g)$ est dans le Levi associ\'e \`a la d\'ecomposition
(\ref{d\'ecomp}). On en d\'eduit donc la formule voulue c'est \`a dire que:
\begin{lemm}\label{lemRes}
\[
\widetilde{\om_{\chi}}_{|U(W)\times U(W)}=\om_{\chi}\otimes\widetilde{\chi}\omega_{\chi}^{\vee}\]
\end{lemm}

\subsubsection{Cas unitaire}
 Il d\'epend du choix
d'un caract\`ere $\widetilde{\chi}$ v\'erifiant:\[
\widetilde{\chi}_{|F^{*}}=\epsilon_{E/F}^{\dim V}\]

Si on se donne deux caract\`ere $\chi_{1}$ et $\chi_{2}$ v\'erifiant
la relation pr\'ec\'edente alors, $\chi_{2}\chi_{1}^{-1}$ d\'efinit un
caract\`ere de $E^{*}$ trivial sur $F^{*}$. Soit:\[
E^{1}=\left\{ z\in E|\ zz^{c}=1\right\} \]
Si $z\in E^{1}$, on peut choisir $u\in E$ tel que:

\begin{eqnarray*}\label{norme}z&=&\frac{u}{u^{c}}\end{eqnarray*}

On d\'efinit un caract\`ere $\alpha_{\chi_{2},\chi_{1}}$ de $E^1$ par:\[
\alpha_{\chi_{2},\chi_{1}}(z)=\chi_{2}\chi_{1}^{-1}(u)\]

qui est ind\'ependant de $u$. La comparaison de la famille de sections
est la suivante:

\[
\om_{\chi_{2}}=(\alpha_{\chi_{2},\chi_{1}}\circ\det)\otimes\om_{\chi_{1}}\]

Remarquons que si l'on se donne un \'el\'ement non nul $\delta\in E$
de trace nulle alors si $z\in E^{1}$, on peut choisir dans $(\ref{norme})$, $u=\delta(z-1)$. 

\subsection{Les r\'esultats de Pan}

On suppose l'extension $E/F$ non ramifi\'ee dans le cas unitaire et
on choisit une unit\'e $\delta\in E$ de trace nulle. Le paragraphe pr\'ec\'edent
rappelait la d\'efinition du cocycle de la repr\'esentation de Weil, cependant
cette repr\'esentation n'\'etait explicite que dans les mod\`eles de Schr\"odinger,
dans le mod\`ele $S(\mathbf{W},\mathbb{S})$ il existe donc un caract\`ere
$\zeta_{V,\chi}$ de $K(L_{W})$ tel que:\[
\om_{\chi}(g)\vhi(w)=\zeta_{V,\chi}(g)\rho(g)\vhi(g^{-1}w)\ \forall g\in U(W)\]

pour tout $g\in K(L_{W})$ le caract\`ere $\zeta_{V,\chi}$ a \'et\'e determin\'e
sous l'hypoth\`ese de non ramification par Pan \cite{Pan1}, le r\'esultat est le suivant:\[
\zeta_{V,\chi}(g)=\chi(\delta(\det g-1))\zeta(g)\]

avec $\zeta(g)$ un caract\`ere d'ordre $2$ de $K(L_{W})$ qui est
trivial d\`es que $W$ est non ramifi\'e.

\begin{lemm}[\cite{Pan1}]
\label{lem:SuppR}Supposons que toutes les donn\'ees soient non ramifi\'ees except\'ee peut-\^etre $V$, alors le caract\`ere $\zeta_{V,\chi}$ est trivial c'est \`a dire que
dans le mod\`ele $S(\mathbf{W},\mathbb{S})$:\[
\om_{\chi}(g)\vhi(w)=\rho(g)\vhi(g^{-1}w)\]
\end{lemm}
Ce r\'esultat sera utilis\'e dans la partie \ref{pararam}.

\section{D\'emonstration du lemme \ref{Nram}}\label{pararam}

On est toujours dans une situation locale, il s'agit de prouver le
lemme. On suppose que la place $v\notin T_{V,\psi,\chi}$ et
que $W$ est ramifi\'e.

\subsection{R\'eduction du probl\`eme}

En suivant Pan on consid\`ere des mod\`eles de la repr\'esentation de Weil
plus g\'en\'eraux que ceux du paragraphe 2. On avait impos\'e aux r\'eseaux
$L_{V}$ et $L_{W}$ d'\^etre presque autoduaux. On dira plus g\'en\'eralement
qu'un r\'eseau $L$ est bon si $\varpi L^{*}\subset L\subset L^{*}$.
Si $L$ et $L'$ sont deux r\'eseaux bons de $V$ et $W$. On d\'efinit
le r\'eseau (bon) suivant de $\mathbf{W}$:\[
A(L,L')=L^{*}\otimes L'\cap L\otimes L'^{*}\]

Comme on est en caract\'eristique impaire c'est un sous-groupe ab\'elien
du groupe de Heisenberg $H(\mathbf{W})$. Rappelons (cf. partie \ref{partie3}) que
l'on peut associer au
choix de ces deux bons r\'eseaux de $V$ et de $W$ un mod\`ele de Waldspurger
g\'en\'eralis\'e $S(\mathbf{W},\mathbb{S})$. Les fonctions invariantes par le sous-groupe
$A(L,L')$ sont exactement les fonctions \`a support dans $B(L,L')=A(L,L')^{*}$
et on a un isomorphisme de $H(\mathbf{b})$-module:\[
\left\{ \begin{array}{ccc}
\mathbb{S} & \rightarrow & S(\mathbf{W},\mathbb{S})^{A(L,L')}\\
s & \mapsto & \vhi_{s}(b)=\rho(b)s\\
\\\end{array}\right.\]

De plus sous ces hypoth\`eses $\om_{\chi}$ restreintes au compact $K(L_{W})$
est bien le scindage usuel (lemme $\ref{lem:SuppR}$). En g\'en\'eral, si
$A$ est un r\'eseau bon d'un espace h\'ermitien $Z$, on d\'efinit les deux
espaces vectoriels sur le corps r\'esiduel de $E$ suivants:
\[
\mathbf{a}=A/\varpi A^{*}\text{ et }\mathbf{a}^{*}=A^{*}/A\]

On a une application de r\'eduction surjective:
$$K(L)\rightarrow U(\mathbf{l})\times U(\mathbf{l}^*)$$
dont nous notons $G_{L,0}^+$ le noyau, sous l'hypoth\`ese de non ramification
(c'est \`a dire que $L=L^*$), on obtient la suite exacte suivante:
\[
1\rightarrow G_{L,0}^{+}\rightarrow K(L)\rightarrow U(\mathbf{l})\rightarrow1\]
et le fait que:
$$\mathbf{b}=\mathbf{l'}^*\otimes\mathbf{l}$$

On notera $\overline{\bullet}$ la fl\`eche de r\'eduction. A l'aide de
la remarque sur le scindage on v\'erifie que pour tout $k\in K(L)$:\[
k\bullet\vhi_{s}=\vhi_{\rho(\overline{k})s}\]

En particulier, $S(\mathbf{W},\mathbb{S})^{A(L,L')}$ est stable par
l'action de $K_{V}(L)$ qui agit par r\'eduction comme
la repr\'esentation de Weil finie. Un des r\'esultats principaux de Pan
est le suivant:\begin{equation}
S^{G_{L,0}^{+}}=\om(\mathcal{H}_{W})\sum_{L'\text{ bon}}S^{A(L,L')}\label{eq:Hecke}\end{equation}

Remarquons d'abord le fait suivant:\[
S^{gA}=\om(g)S^{A}.\]

Il suffit donc de regarder dans la somme les r\'eseaux modulo l'action de
$U(W)$. A cette action pr\`es les r\'eseaux bons
sont faciles \`a d\'ecrire. Soit une d\'ecompostion de Witt $W={X\oplus W}_{0}\oplus Y$
et une base $x_{1,}\dots,x_{r}$ de $X$, soit $y_{1},\dots,y_{r}$
la base duale on suppose de plus que $W_{0}$ est anisotrope. Il existe
alors un unique r\'eseau bon $A$ de $W_{0}$ et tout r\'eseau bon de
$W$ est \`a conjuguaison pr\`es de la forme:

\[
N_{i}=\varpi\mathcal{O}x_{1}\oplus\dots\oplus\varpi\mathcal{O}x_{i}\oplus\mathcal{O}x_{i+1}\dots\oplus\mathcal{O}x_{r}\oplus A\oplus\mathcal{O}y_{1}\oplus\dots\oplus\mathcal{O}y_{r}\]

On a alors:\[
\mathbf{b_{i}}=\mathbf{n}_{i}^{*}\otimes\mathbf{l}\]

D'apr\`es les r\'esultats pr\'ec\'edents il s'agit de d\'emontrer que pour tout
$i$, sous l'hypoth\`ese que $\mathbf{a}^{*}$ est non nul (c'est \`a
dire que $W$ est ramifi\'e), la repr\'esentation $\mathbb{S}_{i}$ de
Weil du groupe fini $\text{Sp}(\mathbf{b}_{i})$ n'admet pas d'invariants
sous l'action de $U(\mathbf{l})$. Nous ne connaissons pas de
r\'ef\'erence pour ce r\'esultat qui doit \^etre classique nous en donnons
une preuve.

\subsection{Repr\'esentation de Weil finie.}

Nous reprennons les notations de la partie pr\'ec\'edente:

\begin{lemm}
Si $\mathbf{a}^{*}$ est non nul alors $\mathbb{S}_{i}$ n'a pas de
vecteur invariant sous $U(\mathbf{l})$.
\end{lemm}
\begin{proof}
On a une d\'ecomposition de Witt de $\mathbf{n}_{i}^{*}=X_{i}\oplus\mathbf{a}^{*}\oplus Y_{i}$
avec $X_{i}$ isotrope de dimension $i$. On utilise pour d\'efinir
$\mathbb{S}_{i}$ un mod\`ele mixte. Soit $\mathbb{S}_{0}$ un mod\`ele
de la repr\'esentation de Weil de $\text{Sp}(\mathbf{l}\otimes\mathbf{a}^{*})$.
On choisit alors pour mod\`ele de la repr\'esentation de Weil:\[
S(\mathbf{l}\otimes X_{i},\mathbb{S}_{0})\]

Les fonctions de $\mathbf{l}\otimes X_{i}$ \`a valeurs dans $\mathbb{S}_{0}$.
L'action de $U(\mathbf{l})$ est donn\'ee par:\[
k\bullet\vhi(x)=\rho_{0}(k)\vhi(k^{-1}x)\]

Soit $(x_{j})_{j\in J}$ une famille de repr\'esentants des orbites
de $\mathbf{l}\otimes X_{i}$ sous l'action de $U(\mathbf{l})$, on
note $H_{j}$ le stabilisateur de $x_{j}$. On a sous l'action de
$U(\mathbf{l})$ un isomorphisme:\[
S(\mathbf{l}\otimes X_{i},\mathbb{S}_{0})=\oplus_{j\in J}\text{ind}_{H_{j}}^{U(\mathbf{l})}\mathbb{S}_{0}\]

En particulier par r\'eciprocit\'e de Frobenius:\[
S(\mathbf{l}\otimes X_{i},\mathbb{S}_{0})^{U(\mathbf{l})}=\oplus_{j\in J}\mathbb{S}_{0}^{H_{j}}\]

Nous commen\c{c}ons par d\'ecrire les diff\'erents groupes $H_{j}$ pouvant
intervenir. Les $H_{j}$ (j variant \'egalement) sont exactement les
fixateurs des sous-espaces vectoriels de $\mathbf{l}$ de dimension
$\leq r$. Soit $\mathbf{m}$ un tel sous-espace vectoriel. Notons
$\mathbf{m'}=\mathbf{m}\cap\mathbf{m}^{\perp},$ le radical de $\mathbf{m}$.
Soit $\mathbf{m_{0}}$ un suppl\'ementaire de $\mathbf{m'}$ dans $\mathbf{m}$.
Soit $\mathbf{m}''$ un sous-espace isotrope de $\mathbf{l}$ en dualit\'e
avec $\mathbf{m}'$. Enfin, notons $\mathbf{l}_{0}$ le suppl\'ementaire
orthogonal du sous-espace non d\'eg\'en\'er\'e $\mathbf{m}'\oplus\mathbf{m}_{0}\oplus\mathbf{m}''$.
On a une d\'ecomposition de Witt suivante de $\mathbf{l}$:\begin{equation}
\mathbf{m'}\oplus(\mathbf{m_{0}\oplus\mathbf{l}_{0}})\oplus\mathbf{m}''\label{eq:dec}\end{equation}

La somme des deux premiers facteurs \'etant \'egale \`a $\mathbf{m}$. Notons $H(\mathbf{m})$
le fixateur de $\mathbf{m}$. La matrice de la forme h\'ermitienne est dans une
base adapt\'ee \`a la d\'ecomposition pr\'ec\'edente:\[
\left(\begin{array}{cccc}
0 &  &  & 1\\
 & x_{0}\\
 &  & y_{0}\\
\epsilon &  &  & 0\end{array}\right)\]

On a alors:\[
H(\mathbf{m})=\left\{ h(u,\delta,\alpha)=\left(\begin{array}{cccc}
1 & 0 & -\epsilon\ ^{t}\overline{\delta}y_{0}u & \alpha\\
0 & 1 & 0 & 0\\
0 & 0 & u & \delta\\
0 & 0 & 0 & 1\end{array}\right)\right\} \]
avec la condition que:
$$\ u\in U(\mathbf{l}_{0})\text{ et }\epsilon\alpha+\ ^{t}\overline{\alpha}+\ ^{t}\overline{\delta}y_{0}\delta=0.$$
On peut alors \`a nouveau d\'ecrire $\mathbb{S}_{0}$ par un mod\`ele mixte
pour lequel la description de l'action de $H(\mathbf{m})$ est bien
connu (\cite{Li3} page 246-247 formule (37),(39) et (41)). Soit $\mathbb{S}_{1}$ un mod\`ele de la repr\'esentation
de Weil de $\text{Sp}(\mathbf{m}_{0}\otimes\mathbf{a}^{*})$ et $\mathbb{S}_{2}$
un mod\`ele de la repr\'esentation de Weil de $\text{Sp}(\mathbf{l}_{0}\otimes\mathbf{a}^{*})$.
On utilise comme mod\`ele:\[
\mathbb{S}_{0}=S(\mathbf{a}^{*}\otimes\mathbf{m}',\mathbb{S}_{1}\otimes\mathbb{S}_{2})\]

L'action de $H(\mathbf{m})$ est d\'ecrite de la fa\c{c}on suivante:

\begin{enumerate}
\item $\rho(h(1,0,\alpha))\vhi(x)=\psi(\text{tr}\alpha(x_{i},x_{j}))\vhi(x)$
\item $\rho(h(1,\delta,\frac{1}{2}\ ^{t}\overline{\delta}y_{0}\delta))\vhi(x)=\rho(\delta\circ x)\vhi(x)$
\item $\rho(h(u,0,0))\vhi(x)=\rho_{2}(u)\vhi(x).$
\end{enumerate}
Soit $\vhi\in\mathbb{S}_{0}^{H(\mathbf{m})}$, alors le premier point
et le fait que $\mathbf{a}^{*}$ soit anisotrope montre que le support
de $\vhi(x)=0$ si $x\neq 0$, le troisi\`eme point
montre que $\vhi(0)\in\mathbb{S}_{1}\otimes\mathbb{S}_{2}^{U(\mathbf{l}_{0})}$.
On a ainsi:\[
\mathbb{S}_{0}^{H(\mathbf{m})}=\mathbb{S}_{1}\otimes\mathbb{S}_{2}^{U(\mathbf{l}_{0})}\]

On est ramen\'e \`a d\'emontrer le fait suivant:\[
\mathbb{S}_{2}^{U(\mathbf{l}_{0})}=\left\{ 0\right\} \]

Supposons dans un premier temps que $\mathbf{l}_{0}=X\oplus Y$ soit
d\'eploy\'e, alors on peut utiliser pour $\mathbb{S}_{2}$ le mod\`ele de
Schr\"odinger $S(X\otimes\mathbf{a}^{*})$. Soit $\vhi$ un vecteur
de $\mathbb{S}_{2}$ invariant par $U(\mathbf{l}_{0})$. L'action
du radical unipotent du parabolique de Siegel associ\'e \`a $X$ est donn\'e
par la m\^eme formule que le point $1$ pr\'ec\'edent. On en d\'eduit donc
que le support de $\vhi$ est concentr\'e en $0$, mais un certain \'el\'ement
de $U(\mathbf{l}_{0})$ agit comme une transformation de Fourier.
La condition de support est donc impossible sauf si $\vhi=0$. Pour
conclure, il suffit de remarquer que dans tous les cas $\mathbf{l}_{0}$ (on
a suppos\'e que dans le cas 1 (resp. le cas 2) que $m>n$ (resp. $m>n+1$)).
\end{proof}
\section{Caract\'eristique r\'esiduelle 2}\label{finitude2}
Le but de cette partie est de d\'emontrer en crarct\'eristique r\'esiduelle $2$
la proposition \ref{Pfin}.

\subsection{Deux lemmes g\'en\'eraux.}

On se donne $W$ un espace $\epsilon$-h\'ermitien par rapport \`a une
extension $E/F$. Si $L$ est un r\'eseau de $W$, on note $L^{*}$
le r\'eseau dual d\'efinit par rapport \`a un id\'eal de valuation
$r_W$. On note $\alpha$ la valuation de $2$ dans $E$.

\begin{lemm}
\label{lem:app1}On se donne $L$ un r\'eseau de $W$ alors il existe
$d\in\N$ tel que pour tout sous-espace isotrope $X\subset W$, on
a une d\'ecomposition de Witt de $W$, $W=X\oplus W_{0}\oplus Y$ de
sorte que:

\[
L\subset\varpi^{-d}(L\cap X\oplus L\cap W_{0}\oplus L\cap Y)\]

\end{lemm}

\begin{proof}

On se donne $c$ tel que:\begin{equation}
\varpi^{c}L^{*}\subset L\text{ et }\varpi^{c}L\subset L^{*}\label{eq:defc}\end{equation}
 D'apr\`es le th\'eor\`eme des diviseurs \'el\'ementaires appliqu\'es
aux deux r\'eseaux $L\cap I$ et $L^{*}\cap I$ de $I$. On a une base
$e_{1},\dots,e_{r}$ de $L\cap I$ et des entiers $\alpha_{1},\dots,\alpha_{r}$
tels que la famille $e_{1}^{*}=\varpi^{\alpha_{1}}e_{1},\dots,e_{r}^{*}=\varpi^{\alpha_{1}}e_{r}$
soit une base de $L^{*}\cap I$. On a pour tout $i$, $|\alpha_i|\leq c$
d'apr\`es $(\ref{eq:defc})$. Comme $I$ est isotrope on a: $I\subset I^{\perp}$.
On peut alors compl\'eter la base $e_{1},\dots,e_{r}$ de $L\cap I$
en une base $e_{1},\dots,e_{s}$ de $L\cap I^{\perp}$. Les r\'eseaux
$\frac{L}{L\cap I^{\perp}}$ et $L^{*}\cap I$ sont en dualit\'e, on
peut donc trouver des \'el\'ements $E_{1,},\dots,E_{r}$ de $L$ v\'erifiant:

\begin{enumerate}
\item $(E_{i,},e_{j}^{*})=\delta_{i,j}\varpi^{r_W}$
\item $e_{1},\dots,e_{s},E_{1},\dots,E_{r}$ est une base de $L$
\end{enumerate}
La matrice de la forme est alors dans la base $e_{1},\dots,e_{s},E_{1},\dots,E_{r}$:\[
\left(\begin{array}{ccc}
0 & 0 & \epsilon\ ^{t}\overline{\beta}\\
0 & w & {\epsilon\ }^{t}\overline{u}\\
\beta & u & Z\end{array}\right)\]
Les diff\'erents coefficients de la matrice sont de
valuation sup\'erieure \`a $r_w-c$.

Si on fait le changement de base suivant:\[
\left(\begin{array}{ccc}
1 &  & x\\
 & 1\\
 &  & 1\end{array}\right)\]

avec $x=-\frac{1}{2}\beta^{-1}Z$, on obtient dans la nouvelle base
comme expression pour la forme:\[
\left(\begin{array}{ccc}
0 & 0 & \epsilon\ ^{t}\overline{\beta}\\
0 & w & {\epsilon\ }^{t}\overline{u}\\
\beta & u & 0\end{array}\right)\]

Et si on fait encore le changement de base:\[
\left(\begin{array}{ccc}
1 & y\\
 & 1\\
 &  & 1\end{array}\right)\]

avec $y=-\beta^{-1}u$, on a alors comme expression pour la forme:\[
\left(\begin{array}{ccc}
0 & 0 & \epsilon\ ^{t}\overline{\beta}\\
0 & w & 0\\
\beta & 0 & 0\end{array}\right)\]
Comme les matrices $x$ et $y$ sont \`a valeurs dans $\varpi^{-c'}\mathcal{O}$
avec $c'$ qui ne d\'epend que de $L$, on en d\'eduit bien le r\'esultat.
\end{proof}

\begin{lemm}
\label{lem:app2}Soit $L$ un r\'eseau de $W$ et $a$ un entier. Il
existe un entier $b$ tel que pour tout sous $\mathcal{O}$-module
$R$ de $L$ v\'erifiant:\[
\forall r_{1},r_{2}\in R\ \langle r_{1},r_{2}\rangle\in\varpi^{b}\mathcal{O}\]
Il existe alors un sous-espace isotrope $I$ de $W$ tel que:\[
R\subset I\cap L+\varpi^{a}L\]

\end{lemm}

\begin{proof}

\begin{Fait}
Pour tout $a\in\mathbb{N}$, il existe $b\in\mathbb{N}$ tel que pour
tout $b'\geq b$ et tout vecteur $e_{1},\dots,e_{r}$ de $L$ lin\'eairement
ind\'ependant dans $L/\varpi L$ et $0\leq \alpha_1,\dots,\alpha_r\leq a$

v\'erifiant:\[
\left\langle \varpi^{\alpha_i} e_{i},\varpi^{\alpha_j}e_{j}\right\rangle \in\varpi^{b'}\mathcal{O}\]
Alors il existe des vecteurs $e_{1}',\dots,e'_{r}$ v\'erifiant:

\begin{enumerate}
\item Pour tout $i\in\left\{ 1,\dots,r\right\} $ $e_{i}-e_{i}'\in\varpi^{a}L$
\item $\left\langle \varpi^{\alpha_i} e_{i}',\varpi^{\alpha_j} e_{j}'\right\rangle \in\varpi^{b'+1}\mathcal{O}$
\end{enumerate}
\end{Fait}
Montrons que le fait suffit pour conclure. Soit $e_{1},\dots,e_{n}$ une base
de $L$ tel que: ($\varpi^{\alpha_{1}}e_{1},\dots,\varpi^{\alpha_{r}}e_{r}$)
soit une base de $R$. Il est clair que l'on peut supposer que pour
tout $i$ , $\alpha_{i}\leq a$. On a alors:\[
\left\langle e_{i},e_{j}\right\rangle \in\varpi^{b'-\alpha_{i}-\alpha_{j}}\]
on construit alors par r\'ecurrence des suites $e_{i}(k)$
de vecteurs de $L$ v\'erifiant:
\begin{enumerate}
\item Pour tout $i\in\left\{ 1,\dots,r\right\} $ $e_{i}-e_{i}(k)\in\varpi^{a}L$
\item $\left\langle \varpi^{\alpha_i}e_{i}(k),\varpi^{\alpha_j}e_{j}(k)\right\rangle \in\varpi^{b+k}\mathcal{O}$
\end{enumerate}
En extrayant des sous-suites convergeantes des suites $e_{i}(k)$,
on peut choisir pour $I$ le sous-espace engendr\'e par les limites qui est
istrope et v\'erifie la relation voulue.$\boxtimes$

Il nous reste \`a d\'emontrer le fait. On a:\[
\left\langle e_{i},e_{j}\right\rangle \in\varpi^{b'-\alpha_{i}-\alpha_{j}}\]
Soit $e_{1}^{*},\dots,e_{n}^{*}$ la base duale de la base $e_{1},\dots,e_{n}$.
On cherche $e'_{i}$ sous la forme suivante:\[
\varpi^{\alpha_i}e_{i}'=\varpi^{\alpha_{i}}e_{i}+\sum_{j=1}^{r}a_{i,j}e_{j}^{*}\]
avec $a_{i,j}\in\varpi^{a+c}\mathcal{O}$. On aura alors bien $(1)$. On veut de plus que l'\'equation
suivante soit v\'erifi\'ee:\[
\left\langle \varpi^{\alpha_i}e_{i}',\varpi^{\alpha_k}e_{k}'\right\rangle \in\varpi^{b'+1}\mathcal{O}\]

Or:\[
\left\langle \varpi^{\alpha_i} e_{i}', \varpi^{\alpha_k}e_{k}'\right\rangle =\varpi^{\alpha_{i}}\overline{\varpi}^{\alpha_{k}}\left\langle e_{i},e_{k}\right\rangle +\varpi^{\alpha_{i}}\overline{a_{k,i}}+\overline{\varpi}^{\alpha_{k}}\epsilon a_{i,k}+\sum_{j,l}a_{i,j}\overline{a_{k,l}}\left\langle e_{i}^{*},e_{j}^{*}\right\rangle \]

Nous posons $a_{i,j}=-\frac{\epsilon}{2}\varpi^{\alpha_{i}}\left\langle e_{i},e_{j}\right\rangle \in\varpi^{b'-\alpha_{j}-\alpha}$,
cela \'elimine les trois premiers termes et on a:\[
\left\langle \varpi^{\alpha_i}e_{i}',\varpi^{\alpha_k}e_{k}'\right\rangle =\sum_{j,l}a_{i,j}\overline{a_{k,l}}\left\langle e_{i}^{*},e_{j}^{*}\right\rangle \]

Cet \'el\'ement appartient \`a $\varpi^{2b'-2(a+\alpha)-c}\mathcal{O}$
et donc que si on pose $b=2(a+\alpha)+c+1$, on obtient le fait.
\end{proof}

\subsection{D\'emonstration du r\'esultat de finitude}

Nous reprennons les notations de la partie 2.3. Soit $r'$ l'entier
tel que $\varpi_{F}^{r'}$ soit le conducteur de $\psi$ et de m\^eme
$r$ l'indice du conducteur de $\psi\circ\tr$. On consid\`ere toujours les inclusions
$A\subset L\subset B$. Comme $L$ est un r\'eseau sympl\'ectique autodual
de $\mathbf{W}$. Il existe une d\'ecomposition de Witt $X\oplus Y$
de $\mathbf{W}$ telle que $L=X\cap L\oplus Y\cap L$. On d\'efinit
si $l=x+y\in L$:\[
\zeta(x+y)=\langle u,v\rangle\]

Pour tout $a\in L$, on a: $\zeta(a+a')-\zeta(a)-\zeta(a')=\langle a,a'\rangle[2\om_{F}^{r'}]$.
On en d\'eduit donc que l'espace suivant:\[
S(\mathbf{W}//L)=\left\{ \vhi:\mathbb{\mathbf{W}}\rightarrow\C|\ \vhi(a+w)=\psi(\frac{1}{2}\zeta(a)+\frac{1}{2}B(w,a))\vhi(w)\text{ }\forall a\in L\right\} \]

est un mod\`ele de la repr\'esentation de Weil muni de l'action de $H(\mathbf{W})$
donn\'ee par:\[
\rho(w',t)\vhi(w)=\psi(\frac{1}{2}B(w,w')+t)\vhi(w+w')\]

Pour tout $w\in\mathbf{W}$, on note $s_{w}$ l'unique fonction \`a
support dans $w+L$ telle que $s_{w}(w)=1$. Pour tout $g\in U(\mathbf{W})$
la fonction $\om(g)s_{w}$ est \`a support dans:\begin{eqnarray}
 & g(w+L)+\frac{1}{2}(gL+L)\label{eq:support}\end{eqnarray}

On cherche \`a d\'ecrire l'action des sous-groupes de congruence de $U(W)$.
Introduisons:\[
K^{1}=\left\{ g\in K_{V}(L_{V})|\ (g-1)B\subset2A\right\} \]
Un \'el\'ement de $k\in K^{1}$ v\'erifie\[
\left\{ \begin{array}{c}
kL=L\text{ et, }\\
\zeta(k^{-1}a)=\zeta(a)[2]\text{ pour tout }a\in L\end{array}\right.\]

L'action du groupe $K^{1}$ sur l'espace $S(\mathbf{W}//L)$ est alors
donn\'ee comme dans le cas de caract\'eristique impair par:\[
\om(k)\vhi(x)=\vhi(k^{-1}x)\]
Remarquons que $K(\varpi^{k}L_V)\subset K^1$ si $k\geq \alpha+1$. Soit $k\geq\alpha+2$, nous notons $K=K(\varpi^{2k}L_{V})$.
Soit $t\geq0$, notons $S[t]$ les fonctions \`a support dans $\varpi^{-(t+k)}L_{W}\otimes L_{V}^{*}$.
$S[t]$ est stable par l'action de $K$. 
\begin{prop} Il existe $t_0\geq 0$ tel que:
\[S^{K}\subset\mathcal{H}_{W}S[t_0]\]
\end{prop}
\begin{proof}
Il suffit de trouver un certain $t_0$, tel que pour tout $t\geq
t_{0}+1$:
\[S[t]^{K}\subset\mathcal{H}_{W}S[t-1].\]
Nous consid\'erons un vecteur \[
\vhi\in{S[t]}^{K}\]
$\vhi$ est alors combinaison lin\'eaire de vecteurs de la forme $\int_{K}\om(k)s_{w}dk$. On peut
donc supposer que \[
\vhi=\int_{K}\om(k)s_{w}dk\]

avec $w\in\varpi^{-(t+k)}L_{W}\otimes L_{V}^{*}$.

Posons:\[
K_{w}=\left\{ k\in K|\ (k^{-1}-1)w\subset2L\right\} \]

On remarque que $K_{w}$ est un sous-groupe de $K$ tel que si $k\in K_{w}$:\[
\om(k)s_{w}=\psi\left(\frac{1}{2}B(w,k^{-1}w-w)\right)s_{w}\]

Pour que le vecteur $\vhi$ soit non nul le caract\`ere pr\'ec\'edent doit-\^etre
trivial. On a donc pour tout $k\in K_{w}$:

\[
\psi\left(\frac{1}{2}B(w,k^{-1}w-w)\right)=1\]
 On va en suivant Waldspurger (\cite{Wald0},\cite{Wald}) construire des \'el\'ements de $K_{w}$.
Notons $\text{Herm}(V)$ l'ensemble des endomorphismes de V v\'erifiant:\[
(cv,v')+(v,cv')=0\ \forall v,v'\in V\]

Si les \'el\'ements $1\pm c$ sont inversibles l'automorphisme $\gamma(c)=(1-c)(1+c)^{-1}$
est dans $U(V)$. Pour tout $x,y\in V$. On construit un \'el\'ement $c_{x,y}\in\text{Herm}(V)$
d\'efinit par:\[
c_{x,y}(w)=x(y,w)-\epsilon y(x,w)\]

\begin{lemm}Supposons $t\geq1$, $x\in\varpi^{t+k}L_{V}$ et $y\in\varpi^{k}L_{V}$. Alors l'\'el\'ement
$\gamma(c_{x,y})$ est bien d\'efini et:
\begin{itemize}
\item $\gamma(c_{x,y})$ est dans $K_{w}$
\item $\psi(\frac{1}{2}B(({\gamma(c)}_{x,y}^{-1}-1)w,w))=\psi(B(w,c_{x,y}w))=\psi\left(\tr(wx,wy)_{W}\right)$
\end{itemize}
O\`u on a identifi\'e $V\otimes W$ \`a $\hom(W,V)$ par l'isomorphisme
$\lambda$ suivant:\begin{equation}
\begin{cases}
\begin{array}{ccc}
V\otimes W & \rightarrow^{\lambda} & \hom(V,W)\\
v\otimes w & \mapsto & (v'\mapsto(v',v)w)\end{array}\end{cases}\label{eq:iso2}\end{equation}

\end{lemm}
\begin{proof}

Posons $c=c_{x,y}$. Sous les hypoth\`eses, il est clair que $\gamma(c)$
est bien d\'efini et que le point $(1)$ est vrai, on a de plus le d\'eveloppement:\[
\gamma(c)^{-1}=1+2\sum_{n\geq1}c^{n}\]

On a donc:\begin{eqnarray*}
\psi(\frac{\langle({\gamma(c)}^{-1}-1)w,w\rangle}{2}) & = & \prod_{n\geq1}\psi(\left\langle c^{n}w,w\right\rangle )\\
\\\end{eqnarray*}

Mais d'apr\`es la d\'efinition de $c$, on a si $n\geq2$:\[
\psi(\left\langle c^{n}w,w\right\rangle )=\psi(\left\langle c^{n-1}w,cw\right\rangle )=1\]
vu que $c^{n-1}w\in L$ et $cw\in L$. On a donc bien le premier
point. Le reste est clair. 
\end{proof}

Remarquons que par l'isomorphisme $\lambda$, on
a:\[
\begin{cases}\label{action}
\lambda(gw)=g\circ\lambda(w)\ \forall g\in U(W)\text{ et,}\\
L_{W}\otimes L_{V}^{*}=\hom(L_{V},L_{W})\end{cases}\]

On a d'apr\`es le lemme pr\'ec\'edent:\[
\langle w(\varpi^{t+k}L_{V}),w(\varpi^{t+k}L_{V})\rangle\subset\varpi^{t+r}\mathcal{O}_{E}\]

Soit $d$ v\'erifiant les hypoth\`eses du lemme \ref{lem:app1} pour
le r\'eseau $L_{W}$ et posons $a=d+2$. Le lemme \ref{lem:app2}
nous donne un nombre $b$ par rapport \`a ce choix de $a$. On suppose que $t_0+r\geq b+1$. On applique
ce lemme et on obtient un sous-espace isotrope $I$ de $W$ tel que :\begin{equation}
w(\varpi^{t+k}L_{V})\subset I\cap L_{W}+\varpi^{d+2}L_{W}\label{eq:rela1}\end{equation}

Appliquons le lemme $(\ref{lem:app2})$ , on obtient une d\'ecomposition
de Witt de $W$ de la forme $W=I\oplus W_{0}\oplus J$ v\'erifiant:\begin{equation}
L_{W}\subset\varpi^{-d}(I\cap L_{W}\oplus W_{0}\cap L_{W}\oplus J\cap L_{W})\label{eq:rela2}\end{equation}

Posons $g=\varpi\text{id}_{I}\oplus\text{id}_{W_{0}}\oplus\varpi^{-1}\text{id}_{J}$.
On v\'erifie \`a l'aide des relations $(\ref{eq:rela1})$ et $(\ref{eq:rela2})$
que:\[
g\circ w(\varpi^{t+k}L_{V})\subset\varpi L_{W}\]

C'est \`a dire que $\omega(g)\vhi$ est \`a support dans $S[t-1]$, d'apr\`es
(\ref{eq:support}) et le lemme \ref{action}. On
en d\'eduit le r\'esultat voulu par r\'ecurrence.
\end{proof}

\end{document}